\documentclass[11pt,a4paper]{article}
\usepackage{fullpage}
\usepackage{hyperref}

\usepackage{graphicx}
\graphicspath{{./figures/}}
\usepackage{subcaption}
\usepackage{multirow}
\usepackage{hyperref}
\usepackage{float}
\usepackage{subfloat}
\usepackage{algorithm}
\usepackage[noend]{algpseudocode}
\usepackage{float}
\usepackage[noend]{algpseudocode}

\makeatletter
\def\algbackskip{\hskip-\ALG@thistlm}
\makeatother

\usepackage{comment} % Necessary to comment several paragraphs at once
\usepackage{booktabs} % Horizontal rules in tables 

\usepackage[utf8]{inputenc} % Required for international characters
\usepackage{csquotes} % Required for international quotes
\usepackage[english]{babel} % International language
\usepackage[T1]{fontenc} % Required for output font encoding for international characters
\usepackage{amsthm}
\usepackage{amsmath,amssymb}
\usepackage{mathdots}

\usepackage{ulem}

\usepackage[pdf]{graphviz}

\usepackage{dot2texi}
\usepackage{tikz}

\usetikzlibrary{arrows,positioning,shapes.geometric}

\usepackage{listings}
\usepackage{color} %red, green, blue, yellow, cyan, magenta, black, white
\definecolor{mygreen}{RGB}{28,172,0} % color values Red, Green, Blue
\definecolor{mylilas}{RGB}{170,55,241}

\newtheorem{lemma}{Lemma}
\newtheorem{definition}{Definition}
\newtheorem{proposition}{Proposition}
\newtheorem{theorem}{Theorem}
\newtheorem{problem}{Problem}

\newtheorem{remark}{Remark}
\newtheorem{example}{Example}[section]
\newtheorem{alg}{Algorithm}

\usepackage{tikz} % For tikz figures (to draw arrow diagrams)
\usepackage{tikz-cd}
\usetikzlibrary{positioning,arrows} % Adding libraries for arrows
\usetikzlibrary{decorations.pathreplacing} % Adding libraries for decorations and paths
\usepackage{tikzsymbols} % For amazing symbols :)
\usepackage{wrapfig}
\newcommand{\ST}[1]{\textcolor{red}{{#1}}}

\begin{document}
\newcommand{\commentout}[1]{}

\title{Estimate the spectrum of affine dynamical systems from partial observations of a single trajectory data }
\author{Jiahui Cheng\thanks{Department of Mathematics, Georgia Institute of Technology. Email: jcheng328@gatech.edu}
\and
Sui Tang \thanks{Department of Mathematics,  University of California Santa Barbara. Email: suitang@ucsb.edu}
}

\maketitle

\begin{abstract}

In this paper, we study the nonlinear inverse problem of estimating the spectrum of a system matrix, that drives a finite-dimensional affine dynamical system, from partial observations of a single trajectory data.  In the noiseless case, we prove  an annihilating polynomial of the system matrix, whose roots are a subset of the spectrum, can be uniquely determined from data.  We then study which eigenvalues of the system matrix can be recovered and derive various sufficient and necessary conditions to characterize the relationship between the recoverability of each eigenvalue and the observation locations. We propose various reconstruction algorithms \footnote{ the code is available in \href{https://github.com/lascride/DynamicalSampling_SpectrumRecoveryforJordanCase}{Jiahui Cheng's Github}} with theoretical guarantees, generalizing the classical Prony method, ESPIRIT, and matrix pencil method. We test the algorithms over a variety of examples with applications to graph signal processing, disease modeling and a real-human motion dataset.  The numerical results validate our theoretical results and demonstrate the effectiveness of the proposed algorithms, even when the data did not follow an exact linear dynamical system.

\end{abstract}

\section{Introduction}

Many physical processes in science and engineering are modeled as linear dynamical systems with a state-space formulation. For example, the linear time-invariant systems are widely used to characterize electrical systems and their properties \cite{hespanha2018linear}. Another common example is provided by the diffusion processes over the graphs, which have found wide applications including modeling rumor propagation in social networks \cite{zhang2014diffusion}, traffic movement in transportation network \cite{deri2016new}, spatial temperature profiles over sensor networks \cite{thanou2017learning}, and neural activities at different regions of the brain \cite{sporns2010networks}. In these applications, the states of the dynamical system at different time instances refer to \textit{signals} of interest.  In practice, a network of sensors is often placed to measure the values of evolving signals with varying locations, and the collected data are called $\textit{samples}$. A fundamental inverse problem is to recover the dynamical system from \textit{samples} of evolving signals.

In the case of known dynamics, the inverse problem reduces to the recovery of the initial state and is also called source localization problem. But in many cases, the dynamics are also unknown and  need to estimate from the data.  Recently,  this type of inverse problem has attracted a lot of attention in the graph signal processing community \cite{ioannidis2018semi,pasdeloup2016characterization,ioannidis2018inference,dong2016learning,thanou2017learning,segarra2017network,mateos2019connecting,pasdeloup2017characterization}: the system matrices of the underlying dynamical systems are related to the topology of the underlying graph, such information is not available in many applications, and needs to be estimated. This is not only for enhancing data processing tasks but also for data interpretability, i.e., the graph topology provides an abstraction for the underlying data dependencies.

The previous methods typically assume the signals are fully observed. However, one may only afford to measure the values of the signals at a subset of coordinates, due to the high cost of building accurate sensors and application-specific restrictions. As a result, there is a possible significant loss of spatial information in each step of data acquisition. The inverse problem becomes in general ill-posed. In particular, in certain situations, we are only able to measure the dynamical system from a single trajectory, since the measurement process results in the destruction or alteration of the system under study. In such scenarios, the exact recovery of the dynamical systems is in general infeasible. However, one may still hope to recover the key information of the dynamical system. 

The spectrum of the system matrix provides valuable information about the underlying dynamical system. For example, for linear autonomous systems, the spectrum plays a crucial role in analyzing the stability of the dynamical system. For random walks or diffusion processes over the graphs, the spectrum reveals the structure information of the underlying graphs \cite{van2003graphs}. This connection has been extensively studied in the field of spectral graph theory.  The notable examples include the complete graphs and finite star-like trees, which are completed determined by the spectrum. That is to say, any graphs with the same spectrum are isomorphic \cite{van2003graphs}. Quoting the sentences in \cite{chung1997spectral}: ``We will see that eigenvalues are closely related to almost all major invariants of a graph, linking one extremal property to another. There is no question that eigenvalues play a central role in our fundamental understanding of graphs."

In this paper, we are interested in recovery of spectrum of the system matrix that generates the dynamical system from partial observations of a single trajectory. Affine dynamical systems are a natural starting point because they have a simple structure yet broad applications, including random walks on graphs \cite{lovasz1993random}, diffusion processes \cite{thanou2017learning}, linear mechanical and electronic systems \cite{hespanha2018linear}, compartmental models in biological modeling such as pharmacodynamics, gene regulation \cite{godfrey1983compartmental, bonate2011pharmacokinetic,holter2001dynamic}. The inverse problem in this setting is nonetheless nontrivial because the solution to such a system depends nonlinearly on the system matrix. We shall begin with a discrete finite dimensional affine system:
\begin{align}
x_{t+1}&=Ax_t+c, t=0,1,2,\cdots,\label{linearsystem0}\\
x_0&=b. \label{state}
\end{align} In \eqref{linearsystem}, $x_t \in \mathbb{C}^d$ is the state of the system at time $t$, the vector $b\in \mathbb{C}^d$ is the unknown initial state and the vector $c\in \mathbb{C}^d$ can be viewed as an unknown control or external force term. The system matrix $A\in \mathbb{C}^{d\times d}$ is unknown. We shall also consider the continuous-times analogue of \eqref{linearsystem0}:
\begin{align}\label{cont}
\dot x(t) &=Ax(t)+c, t\geq 0 \\
x(0)&=b;
\end{align} 

 We then formulate the inverse problem in the most general setting, which we call the \textit{dynamical sampling} problem, as follows:

\begin{problem}[Dynamical sampling problem]\label{ds1}
Suppose that we observe the affine dynamical system at time instances $ \tau$, for each $t\in \tau$, only part of the state $x_t$ is observed, $\{x_t(i): i \in \Omega_t \}$ where $\Omega_t\subset \{1,\cdots, d\}$, under what conditions on observed locations $\{\Omega_t\}_{t\in \tau}$, the initial condition $b$ and the control term $c$ such that the key parameters 
of $A$ can be recovered from the space-time samples? If so, what are algorithms that perform efficient reconstructions? 
\end{problem}

Problem \ref{ds1} exhibits features that are similar to many fundamental problems in the interface of signal processing, machine learning, and control theory of dynamical systems: observability of the dynamical systems, network topology identification, super-resolution, deconvolution, completion of the low-rank matrices. However, even in the most basic cases, the dynamical sampling problems are different and necessitate new theoretical and algorithmic techniques. 

In this paper, we investigate the case where $\Omega_t \equiv \Omega \subset [d]:=\{1,2,\cdots,d\}$ and restrict our attention to the recovery of eigenvalues of $A$. We define by $S_{\Omega}$ the observation matrix that 
\begin{align}\label{traj}
 S_{\Omega}x_t=\sum_{i\in \Omega}x_t(i)e_i
 \end{align}where $\{e_i\}_{i=1}^{d}$ is the standard orthonormal basis in $\mathbb{C}^d$. Given the partial observation of a single trajectory \begin{align}
 \{S_{\Omega}x_t, t=0,1,\cdots\},
 \end{align} we develop theory and algorithms for solving the inverse problem of recovering eigenvalues of $A$. 
 
\begin{figure}[!h]\hspace{3mm}
    \minipage{0.25\textwidth}
    \includegraphics[width=\linewidth]{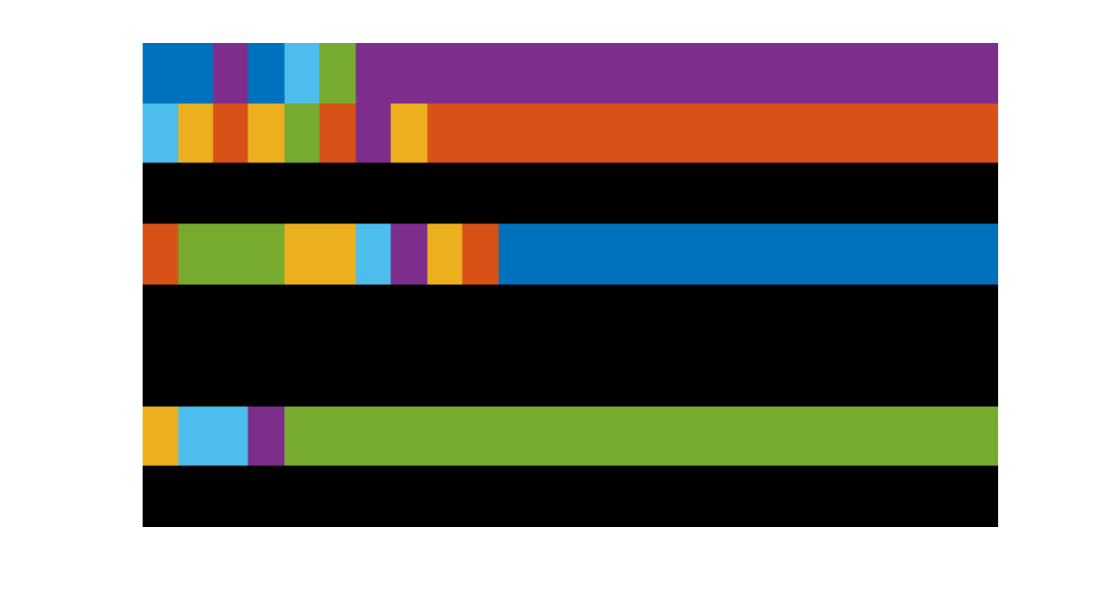}
    \endminipage\hspace{9mm}
    \minipage{0.25\textwidth}
    \includegraphics[width=\linewidth]{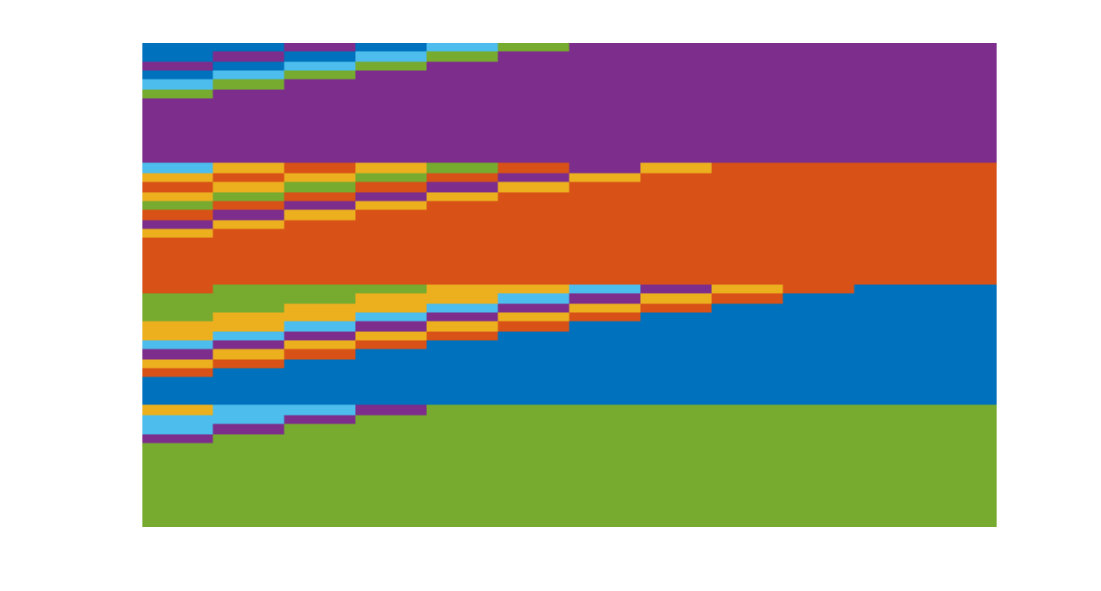}
    \endminipage\hspace{20mm}
    \minipage{0.12\textwidth}
    \centering
    $\hat \lambda_1,\hat \lambda_2,\dots,\hat \lambda_r$
    \endminipage
\end{figure}
\vspace{-5mm}
\begin{tikzpicture}[>=latex']
    \tikzset{block/.style= {draw, rectangle, align=center,minimum width=2cm,minimum height=1cm},
    rblock/.style={draw, shape=rectangle,rounded corners=0.5em,align=center,minimum width=2cm,minimum height=1cm},
    input/.style={ % requires library shapes.geometric
    draw,
    trapezium,
    trapezium left angle=60,
    trapezium right angle=120,
    minimum width=2cm,
    align=center,
    minimum height=1cm
}
    }
    \node [rblock] (observation) {Partial Observations\\of Dynamical System};
    \node [block, right = 1.4cm of observation] (prony) {Prony-like Method};
    \node [rblock, right = 1.4cm of prony] (recovery) {Recovered Eigenvalues\\ of Operator};
    \path[draw,->] (observation) edge (prony)
                   (prony) edge (recovery)
                ;
\end{tikzpicture}
 \subsection{Summary of contributions}
 
We begin by showing that, an annihilating polynomial of $A$ related to $\Omega$ and system parameters can be uniquely determined from \eqref{traj}, whose roots are a subset of eigenvalues of $A$. We then derive necessary and sufficient conditions on the interplay between $\Omega$ and the system parameters to characterize which eigenvalues of $A$ are $\Omega$-recoverable. In particular, we provide characterizations on the universal construction of $\Omega$ that allows the recovery of all eigenvalues of $A$ almost surely. Our theorems shed light on the minimal cardinality of $\Omega$ and guide  constructions of $\Omega$ to recover the target eigenvalues. For numerical algorithms, our proofs  provide a Prony-type method to reconstruct eigenvalues. We also generalize the classical matrix pencil method and estimation of signal parameters via rotational invariant techniques (ESPRIT) to recover eigenvalues and provide theoretical guarantees. Finally, we perform a systematic numerical study to examine the accuracy of the reconstruction and compare the performance of the proposed algorithms for various observational parameters on a variety of examples. We also test the effectiveness of our approach on a real human motion data set.

In summary, the main contributions of this paper are (i) characterizing the uniqueness of annihilating polynomial and the relationship between $\Omega$ and recoverable eigenvalues for affine systems; (ii)characterization of universal constructions of observational locations (iii) proposing various algorithms based on the Prony-method, matrix pencil, and ESPRIT with superior numerical performance.

Our work is built upon recent  progress in studying dynamical sampling problems  where $A$ is known and the goal is to recover the initial state (\cite{aldroubi2017dynamical,tang2017universal,aldroubi2020phaseless}, see \cite{aldroubi2013dynamical,aldroubi2015finite, aldroubi2015exact,aldroubi2017iterative, cabrelli2020dynamical, aldroubi2019frames, christensen2019frame} for recent developments) and the work on system identification aspect \cite{aldroubi2014krylov} and \cite{tang2017system}. In \cite{aldroubi2014krylov}, Aldroubi \textit{et al} studied the homogeneous case when $c=0$ and $A$ is diagonalizable.  They propose Prony-type methods but no numerical examples were presented. In this work, we extend their analysis to affine systems with general system matrices, and propose various algorithms. This extension allows broad applications, such as random walks and diffusion process on directed graphs, and compartmental model in disease modeling. We further derive results for universal selection of observation locations and the continuous-time systems

During the finalization of our work, we noticed the recent work \cite{beinert2021phase} where they propose approximate Prony method to recover the eigenvalue of $A$ in the homogeneous linear dynamical system when the data samples lose the sign information. They prove results on uniqueness for the cases where $A$ is diagonalizable and the eigenvalues of $A$ are collision-free. Our results hold for arbitrary matrix $A$ and we allow the multiplicities of eigenvalues greater than 1. We would expect that our method can be extended to the phaseless samples case.

%In XX, they provide sufficient condition on the linear functional $\mathcal{F}$ such that the desired set of eigenvalues can be recovered. Compared to their results,  we provide general characterization on relationship among system parameters, $f$ and $L$ and recoverable eigenvalues of $A$. Our analysis imply that the sufficient conditions on  functional $\mathcal{F}$ proposed in XX is also necessary and our theories provide guidance on how to design $f$ to satisfy the condition. 

%In the algorithmic level, we propose Prony-typed methods and also generalize the idea of its stabilized variants. We also consider noisy data. We test the performance of our algorithms on both synthetic data and real data sets. 

%Identifiability analysis of dynamical systems has been an area of intense study [2, 4, 7, 28, 30, 38] and it is usually employed to aid model development,

\subsection{Connection with other fields/Related work}
\paragraph{Connection to system identification problem}  Consider a time-invariant linear dynamical system
\begin{align}
x_{t+1}&=Ax_t+Bc_t\\
y_t&=Cx_t
\end{align} where $c_t\in \mathbb{R}^d$ is the input vector and $y_t \in \mathbb{R}^s$ is the output vector. The parameter estimation problem considered in control theory aim to recover the parameter matrices $A, B \in \mathbb{R}^{d\times d}, C\in \mathbb{R}^{s\times d}$ from the output vectors $\{y_t\}$, provided a set of input vectors $\{x_t\}$. The classical results show that $C$ has to be full-rank (i.e, the states are all observed) to make identifiability of parameter matrices possible, see the survey in  \cite{viberg1995subspace} as well as the introduction section in \cite{coutino2020state}.  In the special case of $B=0$, and $C=I$, i.e., $\Omega_t\equiv \{1,\cdots,d\}$ in our setting,  the identifiability of $A$ from  a single trajectory  with a fixed initial condition  has been studied in \cite{stanhope2014identifiability} and later been studied in  affine  dynamical system ($B=I, C=I, c_t\equiv c$) in \cite{duan2020identification}.  It has been show the sufficient and necessary condition is that $\{x_0,Ax_0,\cdots, A^Lx_0\}$ needs to span $\mathbb{C}^d$, which means that $A$ has only one Jordan block for each of its eigenvalues and impose constraints on $x_0$, we refer to \cite{stanhope2014identifiability,duan2020identification} for more details.  This tells us why in the case of partial observations from single trajectory, one should not hope for the full reconstruction of matrix $A$ in general.  It is also mentioned in \cite{duan2020identification} that one can predict the time-dependence for any observable component of the state variable, as long as a  Hankel matrix constructed from data is invertible.   This claim now becomes a special case of Proposition 2.1 developed in this paper. Our result further characterizes when this Hankel matrix is invertible. 
\paragraph{Connection to structured signal recovery problem} Parameter estimation problems of structured signals have been extensively studied in  signal processing. In \cite{peter2013generalized}, the authors present a abstract formulation : let $V$ be a normed vector space over $\mathbb{C}$ and $A$ be a \textit{known} linear operator from $V$ to $V$, one is interested in recovering a signal $b\in V$ that is $M$-sparse with respect to eigenfunctions of $A$
\begin{align}\label{eigenexpansion}
b=\sum_{j\in J}c_jv_j, \text{ with }  |J|=M.
\end{align} The goal is to recover $\{c_j\}$ and $\{v_j\}$ from data samples $\mathcal{F}(A^lb)$ for $l=0,1,\cdots,L$ where $\mathcal{F}:V\rightarrow \mathbb{C}$ is a linear functional. The specific instances include the super-resolution, blind deconvolution, recovery of signals with finite rate innovations, and we refer to \cite{peter2013generalized} for more details.  The keys behind successful recovery of the structured signal $b$ are (1) the eigenvalues $\{\lambda_j\}_{j\in J}$ of $A$ corresponding to $\{v_j\}_{j\in J}$ can be recovered from data under suitable assumptions on $\mathcal{F}$ (2) $A$ is known and its eigenvalues are assumed to have geometric multiplicity 1, therefore finding eigenvectors suffices to finding their corresponding eigenvalues.  In \cite{peter2013generalized}, the authors proposed generalized Prony method to recover the eigenvalues for $L=2M-1$. %

%One typical example is to recover the signal of the form
%$$f(t)=\sum_{j=1}^{d}b_je^{t\omega_j}$$
%with unknown complex parameters $b_j$ and $\omega_j$, $j=1,\cdots,d$ from a set of measurement values $f(l)$ for $l=0,\cdots,L$. This problem can be also formulated as the  inverse problem in a linear dynamical system considered in this paper by taking the matrix $A=\mathrm{diag}(e^{\omega_j})_{j=1,\cdots,d}$ in \eqref{linearsystem},  $b=(b_j)_{j=1,\cdots,d}$ in \eqref{state} and $c=0$.  In XX, the authors consider 

Back to our setting,  assume that  $A$ is diagonalizable and $c=0$ in \eqref{linearsystem0}, and we take $\Omega_l\equiv \{e_i\}$ for some $i \in [d]$  so that  the observational functional  $\mathcal{F} (x_l)=\langle x_l,e_i\rangle=x_l(i)$.  In this case, we use eigenvectors of  $A$ as  basis  and represent the initial condition $b$ as in \eqref{eigenexpansion}. Then the structured signal recovery problem becomes a special case of the dynamical sampling problem considered here. The Theorem \ref{dyn-sampling} developed in this paper generalizes the main results Theorem  2.1 and Theorem 2.3 in \cite{peter2013generalized} in two aspects: (i) our result implies the sufficient conditions on $\mathcal{F}$ is  also necessary (ii) our result provided various sufficient and necessary characterizations to  answer which eigenvalues of $A$ are $\Omega$-recoverable for arbitrary $\Omega \subset \{1,2,\cdots,d\}$.

\iffalse
 given data samples of an (continuous-time) affine dynamical system \eqref{linearsystem}
\begin{align}
\langle x_0,e_i\rangle,\cdots, \langle x_{L_i},e_i\rangle, i \in \mathcal{I} \subset [d]
\end{align}
under what conditions on $\mathcal{I}$ and $L_i$ for $i\in \mathcal{I}$  that  all eigenvalues of $A$ can be recovered from the data, in what way they depends on the system parameters $A$,$b$ and $c$. 
\fi

%

 %In this example, if we take $\Omega_{l}\equiv\{1\}$, then we are only able to recover the eigenvalue $e^{\omega_1}$  (also $\omega_1$)and $b_1$ for any $L\geq 1$, as long as $b_1\neq 0$. However, if we choose $\Omega_{l}=\{(l+1) \mod d\}$ for $l=0,1\cdots,2d-1$,  then we can recover  all the eigenvalues of $A$ and $b$ provided that $b$ has nonzero entries. 

\subsection{Notation}

In the following, we use standard notations. By $\mathbb{N}$, we denote the set of all positive integers. For a positive integer $d$, we use the notation $[d]$ to represent the set $\{1,2,\cdots,d\}$. The linear space of all column vectors with $M$ complex components is denoted by $\mathbb{C}^M$. The linear space of all complex $M \times N$ matrices is denoted by $\mathbb{C}^{M\times N}$.  For a matrix $A=(a_{ij}) \in \mathbb{C}^{M\times N}$, its transpose is denoted by $A^T$, its conjugate-transpose by $A^*$ and its Moore-Penrose pseudoinverse by $A^+.$   For a vector $z=(z_i) \in \mathbb{C}^M$, the $M \times M$ diagonal matrix built from $z$ is denoted by $\mathrm{diag}(z)$. Further,  we use submatrix notation similar to that of MATLAB. For example, if $A \in \mathbb{C}^{M, M+1}$, then ${A}(1 : M, 2 : M + 1)$ is the submatrix of $A$ obtained by extracting rows 1 through $M$ and columns 2 through $M + 1$, and $A(1 : M, M + 1)$ means the last column vector of $A$.

\subsection{Preliminaries}

Throughout the paper,  we assume that the system matrix  $A \in \mathbb{C}^{d\times d}$ and it has distinct eigenvalues $\lambda_1,\dots,\lambda_n$. Consider the Jordan decomposition $A=UJU^{-1}$, where $U \in \mathbb{C}^{d\times d}$ is invertible and the Jordan matrix $J\in  \mathbb{C}^{d\times d}$ is  block diagonal  defined as follows:	\begin{equation}\label{Jordan1}
	    J=\left(\begin{array}{cccc}
	    J_{1} &   O   &   \dots   &   O   \\
	    O   &   J_{2} &   \dots   &   O   \\
	    \vdots& \vdots& \ddots  &   \vdots\\
	    O   &   O   &   \dots   &   J_{n} \\
	    \end{array}\right).
	\end{equation} In \eqref{Jordan1}, for $s=1,\cdots,n$, the Jordan block $J_s$ corresponds to $\lambda_s$ and  $J_s = \lambda_s I_s +  N_s$ where $I_s$ is the identity matrix of dimension $h_s$, and $N_s$ is a nilpotent block-matrix of dimension $h_s$:
\begin{equation}\label{Nil1}
N_s= \left( \begin{array}{cccc}
N_{s_1} & 0 & 0 &0 \\
0 & N_{s_2}& 0&0 \\
0 & 0 & \ddots&0 \\
0&0&0&N_{s_{r_s}}\end{array} \right)
\end{equation}
where each $N_{s_i}$ is a $t_{i}^{(s)}\times t_{i}^{(s)}$ cyclic nilpotent matrix of the form \eqref{Nil2}, \begin{equation}\label{Nil2}
N_{s_i}=\left( \begin{array}{cc}
0 & 0 \\
I_{t_i^{(s)}-1}& 0
\end{array} 
\right) = \left( \begin{array}{cccccc}
0 & 0 &\cdots & 0 &0&0\\
1& 0&\cdots& 0&0&0\\
0& 1&\cdots& 0&0&0\\
\vdots & \vdots & \ddots&\vdots&\vdots&\vdots\\
0&0&\cdots&1&0&0 \\
0&0&\cdots&0&1&0
\end{array} 
\right)
\end{equation} with $t_{1}^{(s)} \geq t_{2}^{(s)}\geq \ldots \geq t_{r_s}^{(s)}$ and $t_{1}^{(s)} + t_{2}^{(s)}+\cdots+ t_{r_s}^{(s)}=h_s$. Also $ h_1 + \ldots + h_n = d$.

Each Jordan block $J_s$ corresponds to an invariant subspace $V_s$ of $A$ and the Jordan form $J$ gives a decomposition of $\mathbb{C}^d$ into invariant subspaces of $A$: $\mathbb{C}^{d}=\oplus_{i=1}^{s}V_i$. We define the projection onto $V_s$ by $P(\lambda_s;A)$. Similarly, we denote by $E_s$ the invariant subspace of $J$ corresponding to the block $J_s$; $E_s$ is spanned by the canonical basis $\{e_i: i=h_1+\cdots+h_{s-1}+1,\cdots, h_1+\cdots+h_{s} \}$ ($h_0=0$ for the case $s=1$); the projection onto $E_s$ is denoted by $P(\lambda_s;J)$ and we have $P(\lambda_s;J)=UP(\lambda_s;A)U^{-1}$. Any vector $f\in \mathbb{C}^d$ admits the unique decomposition $f=\sum_{s=1}^n f_s$, where $f_s=P(\lambda_s;J)f \in E_s$.

Note that $J_{\big|E_s}=\lambda_s+N_s$, 
we will use $Jf_s=(\lambda_s+N_s)f_s$ by viewing the vector $f_s \in \mathbb{C}^{h_s}$ and interpreting the vector $(\lambda_s+N_s)f_s$ as its canonical embedding in $\mathbb{C}^d$.

With the abuse of notation, if $N_s\equiv 0$ for $s=1,\cdots,n$, then $J$ in \eqref{Jordan1} reduced to  a diagonal matrix of the form 

\begin{equation}\label{D1}
	    D=\left(\begin{array}{cccc}
	    D_{1} &   O   &   \dots   &   O   \\
	    O   &   D_{2} &   \dots   &   O   \\
	    \vdots& \vdots& \ddots  &   \vdots\\
	    O   &   O   &   \dots   &   D_{n} \\
	    \end{array}\right).
	    \end{equation}

In this case, $A$ is diagonalizable.

\begin {definition} \label {defJord} Let $k_i^{s}$ denote the row index corresponding to the entry 1 in the last nonzero column of the block $N_{s_i}$ \eqref{Nil2} from the matrix $J$ \eqref{Jordan1}, and let $e_{{k}_{i}^{s}}$ be the corresponding elements of the standard basis of $\mathbb{C}^n$, so that each  $e_{{k}_{i}^{s}}$ is the cyclic vector associated to $N_{s_i}^*$. We also define  $W_s=span\{e_{{k}_{i}^{s}}: i=1,\ldots, r_s\},$ for $s=1,\ldots,n$, and $P_s$ will denote the orthogonal projection onto
$W_s$. The family $ P_J=\{P_j: j=1\ldots n\}$ comprised of these projections will be called the \emph{penthouse family} of the matrix $J$. %Finally,   
\end {definition}

Let us give an example to illustrate the definition above. Consider a Jordan matrix as
\begin{equation}
J = \begin{pmatrix}J_1&\\&J_2\end{pmatrix}=\left(\begin{array}{ccc ccc}3&& &&&\\ 1&3& &&&\\ &1&3 &&&\\ && &3&&\\ && &1&3&\\ && &&&2 \end{array}\right).
\end{equation}
It has 3 nilpotent blocks given as $N_{11}$, $N_{12}$ and $N_{21}$. The cyclic vectors associated to $N_{11}^*$, $N_{12}^*$ and $N_{21}^*$ are $e_3$, $e_5$ and $e_6$, respectively.

		% Definition of minimal polynomial

		\begin{definition} We introduce three kinds of minimal polynomials that are useful in the paper. 
 \begin{itemize}
		\item For $A \in \mathbb{C}^{d\times d}$, the minimal polynomial $q^A$ is the monic polynomial of the smallest degree, such that $q^A(A)\equiv0$, and we denote its degree as $r^A = deg(q^A)$. 
		\item  For any matrix $S\in \mathbb{C}^{m\times d}$, the $S$-altered minimal polynomial of $A$, denoted by $q_{S}^{A}$, is the monic polynomial of smallest degree among all the polynomials p such that $Sp(A)$ = 0 and  $r_{S}^{A}:=deg(q_{S}^{A})$
		\item The $A$-minimal polynomial $q_b^A$ for a vector $b$ in $\mathbb{C}^{d}$ is the monic polynomial of the smallest degree, such that $q_b^A(A)b\equiv0$, and we denote its degree as $r^A_b = deg(q^A_b)$.
\end{itemize} 
\end{definition}
	
\begin{definition}  A Krylov space of order $r$ generated by $A \in \mathbb{C}^{d\times d}$ and $b \in \mathbb{C}^d$, is defined by 	\begin{equation}\mathcal{K}_r(A,b) := span\{b,Ab,\dots,A^{r-1}b\}.\end{equation}
In particular, for any $r\geq r_b^A-1$, we will denote $\mathcal{K}_r(A,b)=\mathcal{K}_{\infty}(A,b)$. 
%An altered Krylov space generated by  $S\in \mathbb{C}^{m\times d}$ and $\mathcal{K}_r(A,b)$, is defined by 
	%\begin{equation}\mathcal{AK}_r(S;A, b) := span\{Sb, SAb,\dots,SA^{r-1}b\}. 
	%\end{equation}
    \end{definition}

	% Definition of (x,A,f)-minimal polynomial 
	\begin{definition}\label{tripleminpoly}
	  Let $S\in \mathbb{C}^{m\times d}, A\in \mathbb{C}^{d\times d}$ and $f\in \mathbb{C}^d$. The $(S,A,b)$-annihilating polynomials are all polynomials $q$ such that $Sq(A)\mathcal{K}_{r_{S}^{A}}(A,b)=\{0\}$. We denoted by $r_{S,b}^{A}$, the smallest degree among all monic  $(S,A,b)$-annihilating polynomials. 
\end{definition}

\iffalse
Given a monic polynomial $q(z)=\sum\limits_{k=0}^{r-1}q_kz^k+z^r$, the companion matrix of $q(z)$ is defined by $$\mathbf{C}^{q(z)}=\begin{bmatrix} 
    0&0 & \cdots&0&-q_0\\
     1& 0 &\cdots&0&-q_1 \\
   0& 1 & \cdots&0 &-q_2\\
\vdots&\vdots&\ddots&\vdots&\vdots\\
 0& 0& \cdots&1& -q_{r-1}
   \end{bmatrix}
 $$
 \fi

	% Definition of Krylov subspace
	\subsection{Solution formulas to affine  systems}
In this section, we present the explicit solution formulas in terms of $A, b$ and $c$ for system \eqref{linearsystem} and its continuous-time counterpart, which will  be useful in our analysis. Let $x_t$ be the solution of discrete system \eqref{linearsystem}. Then using the recursive relation, we obtain that 
\begin{align}\label{sol1}
x_t =A^tb+(A^{t-1}+\cdots+I)c, t=1,2,\cdots, 
\end{align}

For the continuous-time system \eqref{cont}, we let $x^{\mathrm{cont}}_t$ be the solution of  system \eqref{cont}. Then the solution can be obtained by differentiating \eqref{cont} and 
solving corresponding linear initial value problem for $\dot{x}_t$:
\begin{align}
x^{\mathrm{cont}}_t= e^{tA}b+g(t;A)c, t\geq 0
\end{align}where 
\begin{align}
g(t;A)&=\sum_{k=0}^{\infty}\frac{t^{k+1}}{(k+1)!}A^k=tI+\frac{t^2}{2}A+\frac{t^3}{3!}A^2+\cdots.
\end{align}
We also list several useful properties of $g(t;A)$:
\begin{align*}
\frac{d}{dt}g(t;A)&=e^{tA};  \quad\quad
Ag(1;A)=g(1;A)A=e^A-I.
\end{align*} In particular, if $A$ is invertible, then $g(1;A)=(e^A-I)A^{-1}$. 

\section{Main Results}\label{sec2:Main-Results}
\subsection{Discrete-time affine dynamical systems}
 \subsubsection{Case $c=0$.}\label{homogeneous}
We will begin with the homogeneous linear dynamical system with $c=0$. In this case, the discrete dynamical system reduces to 
\begin{align}\label{HDS}
x_t=A^tx_0, x_0=b.
\end{align} 
 
Given the observation locations $\Omega\subset [d]$, we show that the minimal $(S_{\Omega},A,b)$-annihilating polynomial $q_{S_{\Omega},b}^A$ can be uniquely recovered from 
$ \{S_{\Omega}(x_t):t=0,1,\cdots,\}$ and the roots of  $q_{S_{\Omega},b}^{A}$ are  eigenvalues  of $A$. Here by uniqueness, we mean that if $S_{\Omega}x_t=S_{\Omega}\tilde x_t$ for $t=0,1,\cdots$, where $\tilde x_t$ is the trajectory of the system \eqref{HDS} with a system matrix $\tilde A$ and an initial state $\tilde b$,  then $q_{S_{\Omega},b}^A=q_{S_{\Omega},\tilde b}^{\tilde A}$. 
 
 \begin{proposition} \label{Prony}Given partial observations of trajectory data determined by $\Omega\subset [d]$     $$S_{\Omega}(x_t), t=0,\cdots,2r-1, r=r_{S_{\Omega},b}^{A},$$
     we construct a Hankel matrix 
     
  \begin{equation}\label{eq:Hankel}
        H : = \begin{bmatrix} S_{\Omega}b &\dots& S_{\Omega} A^{r-1}b\\
        \vdots&\cdots&\vdots\\
     S_{\Omega}A^{r-1}b &\dots& S_{\Omega}A^{2r-2}b \end{bmatrix}.
    \end{equation}
 Then $H$ is of full column rank and there exist a unique solution  $q=[q_1,\cdots,q_r]^T \in \mathbb{C}^{r}$ to the linear system
  
  \begin{equation}\label{eq:Prony}
        {H}q = - h_{\Omega,r}, {h}_{\Omega,r} = \begin{bmatrix} S_{\Omega} A^rb \\ \vdots\\  S_{\Omega} A^{2r-1}b \end{bmatrix}.
            \end{equation}
We have that,  $q_{S_{\Omega},b}^A(z)=z^{r}+\sum_{i=1}^{r}q_{i}z^{i-1}$. In addition,  the roots of $q_{S_{\Omega},b}^A$ are eigenvalues of $A$. 
\end{proposition}

\begin{proof} Suppose that $Hq=-h_r$. Let the polynomial $q(z)=z^r+\sum\limits_{k=0}^{r-1}q_{k+1}z^k$. Then it follows that 

$$S_{\Omega}q(A)A^tb=0,t=0,\cdots,r-1.$$ By Lemma \ref{lemma3}, the solution is unique and $q=q_{S_{\Omega},b}^A$. It follows from the equation \eqref{Smindegree} in Lemma \ref{lemma2} (see Appendix),  that the roots of $q_{S_{\Omega},b}^A$ are eigenvalues of $A$. 

\end{proof}

Proposition \ref{Prony} in fact provides us with a Prony-type algorithm to reconstruct the annihilating polynomial $q_{S_{\Omega},b}^A$ from data, by solving the Hankel-type equation \eqref{eq:Prony}. Then we find roots of $q_{S_{\Omega},b}^A$ which are a part of eigenvalues of $A$. However,  a key question still not addressed  is that which eigenvalues of $A$ can be recovered. Our goal is to find  the relationship between $\Omega$ and recoverable eigenvalues. Such results are useful in the selection of $\Omega$ to recover the target eigenvalues.

\begin{theorem}\label{dyn-sampling} Assume that the evolution matrix $A \in \mathbb{C}^{d\times d}$ and its Jordan decomposition can be written as $A=UJU^{-1}$ where $J$ is a Jordan matrix as in \eqref{Jordan1}.  Let $b \in \mathbb{C}^d$ be the initial state. Then the polynomial $q_{S_{\Omega},b}^A$ can be  uniquely determined from $\{S_{\Omega}x_t: t=0,1,\cdots\}$. In addition, the following statements are equivalent: for $s=1,\cdots,n,$
\begin{itemize}
\item[(1)] $\lambda_s$ is \textbf{not} a root of $q_{S_{\Omega},b}^A$. 
%\item[(b)]$b_s\perp\sum_{i\in \Omega}\mathrm{span}\{\hat u_{s,i},N_s^*\hat u_{s,i},\cdots, {N_s^*}^{m_{s,i}-1}\hat u_{s,i}: i\in \Omega\}$ where $b_s=P(\lambda_s; J)U^{-1}b$, $\hat u_{s,i} = P(\lambda_s; J)U^*e_i$ and $m_{s,i}=r_{\hat u_{s,i}}^{N_s}$. 
\item [(2)]$(U^{-1}b)_s\perp \mathcal{K}_{\infty}(N_s^{\ast},(U^*e_i)_s)$ for all $i\in \Omega$, where $(U^{-1}b)_s=P(\lambda_s; J)U^{-1}b$, $(U^*e_i)_s = P(\lambda_s; J)U^*e_i$. 

%\item[(c)] $P^*(\lambda_s;A)e_i \perp \mathrm{span}\{ b, Ab,\cdots, A^{r_{b}^A-1}b\} $ for all $i\in \Omega$.
\item [(3)]$P^*(\lambda_s;A)e_i \perp \mathcal{K}_{\infty}(A,b) $ for all $i\in \Omega$, where $P^*(\lambda_s;A)$ denotes the adjoint operator of $P(\lambda_s;A)$.
%\item[(d)] $b \not\perp P(\lambda_s^*; A^*)\mathrm{span}\{e_i,A^{*}e_i,\cdots (A^{*})^{r_{e_i}^A-1}e_i: i\in \Omega\} $

\end{itemize}

\end{theorem}

\begin{proof} The proof is based on Lemma \ref{lemma1}-Lemma \ref{lemma3} in the Appendix. 
\begin{itemize}
\item  Claim 1: the polynomial $q_{S_{\Omega},b}^A$ has $\lambda_s$ as one of its roots if and only if $\mathrm{deg}(q_{S_{\Omega}U,(U^{-1}b)_s}^J):=r_{S_{\Omega}U,(U^{-1}b)_s}^{J} \geq 1$.
 
  On one hand, $q_{S_{\Omega},b}^A=q_{S_{\Omega}U,U^{-1}b}^J$.  On the other hand,  by Lemma \ref{lemma2}, we have $q_{S_{\Omega}U,U^{-1}b}^J=\prod_{s=1}^{n}q_{S_{\Omega}U,(U^{-1}b)_s}^J$ where $(U^{-1}b)_s=P(\lambda_s;J)U^{-1}b$ and 
  $q_{S_{\Omega}U,(U^{-1}b)_s}^J(z)=(z-\lambda_s)^{r_{S_{\Omega}U,(U^{-1}b)_s}^{J}}$.  Therefore  Claim 1 is proved. 
 
 \item Claim 2:  $r_{S_{\Omega}U,(U^{-1}b)_s}^{J}=0$ if and only if $$ (U^{-1}b)_s \perp \mathrm{span}\{{N_s^{*}}^t \hat u_{s,i},  t=0,1,\cdots\}=\mathrm{span}\{{N_s^{*}}^t \hat u_{s,i},  t=0,1,m_{s,i}-1\}, i \in \Omega,$$ where $\hat u_{s,i}= (U^*e_i)_s$, and $m_{s,i}=r_{\hat u_{s,i}}^{N_s^*}$. \\
By Definition \ref{tripleminpoly}, $q_{S_{\Omega}U,(U^{-1}b)_s}^J$ satisfies 

 	\begin{equation}\label{minimalEquationhere}
		S_{\Omega}U q_{S_{\Omega}U,(U^{-1}b)_s}^J(J)J^{t}(U^{-1}b)_s =0, t=0,1,\cdots, r_{(U^{-1}b)_s}^J-1. 	
       \end{equation}
  Note that $(U^{-1}b)_s \in E_s:=P(\lambda_s;J)$ and therefore $J_{\big|E_s}=\lambda_s+N_s$. By the proof of Lemma \ref{lemma2}, we have $q_{S_{\Omega}U,(U^{-1}b)_s}^J(J)_{\big|E_s}=N_s^{r_{S_{\Omega}U,(U^{-1}b)_s}^{J}}$, the equations in \eqref{minimalEquationhere} are  equivalent to 
\begin{equation}\label{minimalEquationnew}
		S_{\Omega}U N_s^{r_{S_{\Omega}U,(U^{-1}b)_s}^{J}} (\lambda_s+N_s)^t(U^{-1}b)_s=0, t=0,1,\cdots, r_{(U^{-1}b)_s}^J-1, \end{equation} which can be further simplified as  
		
		%$\hat u_{s,i}=P(\lambda_s;J)U^*e_i$, which can be further formulated as  
\begin{equation}\label{minimalEquationnew1}
		S_{\Omega}U N_s^{r_{S_{\Omega}U,(U^{-1}b)_s}^{J}} N_s^t (U^{-1}b)_s=0, t=0,1,\cdots, r_{(U^{-1}b)_s}^J-1. 
\end{equation}

Representing  each equation in \eqref{minimalEquationnew1} using the inner product, we obtain that, for $i \in \Omega$
\begin{align}\label{equation}
 \hat u_{s,i} \perp \mathrm{span}\{N_s^t (U^{-1}b)_s,  t=0,1,\cdots, r_{(U^{-1}b)_s}^J-1\},
 \end{align}
  where $\hat u_{s,i}=P(\lambda_s;J)U^*e_i$. 

Note that $\mathrm{span}\{N_s^t (U^{-1}b)_s,  t=0,1,\cdots, r_{(U^{-1}b)_s}^J-1\}=\mathrm{span}\{N_s^t (U^{-1}b)_s,  t=0,1,\cdots \}$, and $\langle \hat u_{s,i}, N_s^t (U^{-1}b)_s\rangle=\langle {N_s^{*}}^t\hat u_{s,i}, (U^{-1}b)_s\rangle$, claim 2 is followed by the fact that $$ \mathrm{span}\{{N_s^{*}}^t \hat u_{s,i},  t=0,1,\cdots\}=\mathrm{span}\{{N_s^{*}}^t \hat u_{s,i},  t=0,1,m_{s,i}-1\},$$ where $m_{s,i}=r_{\hat u_{s,i}}^{N_s^*}$ is the least integer $m$ such that 
${N_s^*}^{m-1}(U^{-1}b)_s\neq 0$ but ${N_s^*}^{m}(U^{-1}b)_s=0$. 

\item Claim 3: the equations in \eqref{minimalEquationhere} are equivalent to 
$$ \langle P^*(\lambda_s;A)e_i,A^t b\rangle=0,t=0,1,\cdots$$

Note that 
\begin{align*}
\langle \hat u_{s,i},J_s^t (U^{-1}b)_s\rangle &= \langle P(\lambda_s;J)U^*e_i,J_s^t (U^{-1}b)_s\rangle\\
&=\langle e_i,UP(\lambda_s;J)J_s^t (U^{-1}b)_s\rangle=\langle e_i,UP(\lambda_s;J)J^t P(\lambda_s;J)U^{-1}b \rangle\\
&=\langle e_i,UP(\lambda_s;J)J^t U^{-1}b \rangle =\langle e_i,UP(\lambda_s;J)U^{-1}UJ^t U^{-1}b\rangle\\
&=\langle e_i, P(\lambda_s;A)A^tb\rangle= \langle P^*(\lambda_s;A)e_i, A^tb\rangle,
\end{align*} where we used the facts that $P(\lambda_s;J)^2=P(\lambda_s;J)$, $P(\lambda_s;J)J=JP(\lambda_s;J)$ and $P(\lambda_s;A)=UP(\lambda_s;J)U^{-1}$. Then the conclusion follows in a similar way as in the proof of Claim 2. 
\end{itemize} 
\end{proof}

In Theorem \ref{dyn-sampling}, we  view $\Omega$  as a set of functionals $\{e_i: i\in \Omega\}$ in the dual space.  Part (2) characterizes the recoverability of the eigenvalue $\lambda_s$ by checking the local orthogonality of  the vector $(U^{-1}b)_s$ with the Krylov subspace generated by   $N_s^*$ and  $\{(U^*e_i)_s: i \in \Omega\}$.  Part (3) provides an equivalent geometric characterization on the global Krylov subspace generated by the  trajectory of  the system and the projection of the vectors $\{U^*e_i: i \in \Omega\}$ onto the invariant subspace $E_s$. In this way, we find the interplay between $A, b$ and $\Omega$ that allows the recovery of the eigenvalue $\lambda_s$.

\begin{remark} Theorem \ref{dyn-sampling} also holds for a diagonalizable matrix $A$. Suppose that $A=UDU^{-1}$ where $D$ is a  diagonal matrix as in $\eqref{D1}$. By setting $N_s=0$, the part (2) in this case can be simplified as 
$$(U^{-1}b)_s\perp (U^*e_i)_s$$ for all $i\in \Omega$, where $(U^{-1}b)_s=P(\lambda_s; D)U^{-1}b$, and $(U^*e_i)_s= P(\lambda_s; D)U^*e_i$. This condition was proved in Theorem 3.7 in \cite{aldroubi2014krylov}.  In particular, when 
$P(\lambda_s;D)$ is a rank one projection, $\lambda_s$ is recoverable if and only if  $(U^{-1}b)_s(U^*e_i)_s \neq 0$ for some $i \in \Omega$, which was listed as a sufficient condition in Theorem 2.1 in  \cite{peter2013generalized} for  $\Omega=\{e_i\}$. 
 \end{remark}

\paragraph{Universal constructions of $\Omega$}
From Theorem \ref{dyn-sampling}, we know that if $(U^{-1}b)_s=P(\lambda_s; J)U^{-1}b=0$, then $\lambda_s$ can not be recovered no matter what choice of $\Omega$ is. However, the set consisting of this kind of initial conditions is of measure zero. Let us   consider a generic set of $\mathbb{C}^d$ defined by $\mathcal{S}=\{b\in \mathbb{C}^d: (U^{-1}b)_s \neq 0 \text{ for } s=1,\cdots,n\}$. If the initial state $b$ is sampled from a non-degenerate probability measure on $\mathbb{C}^d$, then $b\in \mathcal{S}$ almost surely.  One may ask: how to choose $\Omega$ such that all eigenvalues of $A$ can be recovered from partial observations of a single trajectory starting from $b_0 \in \mathcal{S}$. This question yields the following definition. 

\begin{definition}[Universal sampling set]
A set $\Omega$ is said to be universal for the system \eqref{HDS} if all eigenvalues of $A$ can be recovered from the partial observation of the trajectory data $S_{\Omega}x_t$ for $t=0,1\cdots,2r_{S_{\Omega},b}^A-1$  with $x_0=b\in \mathcal{S}$. 

\end{definition}

 Based on  Theorem \ref{dyn-sampling},  we derive various characterizations  on universal sampling sets $\Omega$ that can guide their constructions. 

\begin{theorem} \label{universal}  Assume that the evolution matrix $A \in \mathbb{C}^{d\times d}$ and its Jordan decomposition can be written as $A=UJU^{-1}$ where $J$ is a Jordan matrix as in \eqref{Jordan1}.  The following statements are equivalent
\begin{itemize}
\item[(1)] $\Omega$ is universal for the system \eqref{HDS}. 
\item[(2)] The set of vectors $\{e_i,A^*e_i,\cdots, (A^*)^{r_{e_i}^A-1}e_i:i\in \Omega\}$ span $\mathbb{C}^d$. In other words, $\sum_{i\in \Omega}\mathcal{K}_{\infty}(A^*;e_i)=\mathbb{C}^d.$
%\item[(c)] The set of vectors $\{(N_s^*)^{t} \hat u_{s,i}:  i\in \Omega, t=0,\cdots, \}$ spans $\mathrm{Range}(P(\lambda_s;J))$ for $s=1,\cdots,n$, where  the vector $\hat u_{s,i} = P(\lambda_s; J)U^*e_i$ 
\item[(3)]The set of vectors $\{P_s U^*e_i,  i \in \Omega\}$ spans  $\mathrm{Range}( P_s)$ for $s=1,\cdots,n$, where $ P_J = \{P_s: s = 1, \cdots, n\}$ is the penthouse family for $J$ introduced in Definition \ref {defJord}, in other words,  $P_s$ is the orthogonal projection onto  the span of cyclic vectors corresponding to the  nilpotent block $N_s^*$. 
\end{itemize}
\end{theorem}

\begin{proof}  Note that $A^*=(U^{*})^{-1}J^*U^*$, we have that 
$$\sum_{i\in \Omega}\mathcal{K}_{\infty}(A^*;e_i)=\mathbb{C}^d \iff \sum_{i\in \Omega}\mathcal{K}_{\infty}(J^*;U^*e_i)=\mathbb{C}^d,$$ which is also equivalent to 
$$\sum_{i\in \Omega}P(\lambda_s^*;J^*)\mathcal{K}_{\infty}(J^*;U^*e_i)=E_s:=\mathrm{Range}(P(\lambda_s^*;J^*)), s=1,\cdots,n.$$
The first equivalence is due to the invertibility of $U$ and the second equivalence is due to the facts that $\sum_{s=1}^{n}P(\lambda_s^*;J^*)=I$ and the projections are mutual orthogonal.

 The equivalence between $(1)$ and  $(2)$ can then  be proved using part (3) of Theorem  \ref{dyn-sampling}. Note that

\begin{align}
\langle P^*(\lambda_s;A)e_i,  A^tb \rangle &=\langle P(\lambda_s^*;A^*)e_i,  A^tb \rangle \nonumber\\
 &=\langle P(\lambda_s^*;A^*)(A^*)^te_i,  b \rangle\nonumber\\
&=\langle (U^*)^{-1}P(\lambda_s^*;J^*)U^*(U^*)^{-1}(J^*)^tU^*e_i , b \rangle\nonumber\\
&=\langle P(\lambda_s^*;J^*)(J^*)^tU^*e_i , U^{-1}b \rangle\nonumber\\
&=\langle P(\lambda_s^*;J^*)(J^*)^tU^*e_i ,  P(\lambda_s^*;J^*)U^{-1}b \rangle, \label{innerp}
\end{align} where we used the facts that $P(\lambda_s^*;A^*)=P^*(\lambda_s;A)$, $P(\lambda_s^*;A^*)A^*=A^*P(\lambda_s^*;A^*)$, and $P(\lambda_s^*;J^*)=P(\lambda_s;J).$

$(2)\implies (1):$ for any $b \in \mathcal{S}$, and any $s \in [n]$, we have that $P(\lambda_s;J)U^{-1}b=P(\lambda_s^*;J^*)U^{-1}b \neq 0$.  From \eqref{innerp}, 
we see that it is impossible to have
$$\langle P^*(\lambda_s;A)e_i,  A^tb \rangle=0 \text{ for  all } i \in \Omega, $$
since $\sum_{i\in \Omega}P(\lambda_s^*;J^*)\mathcal{K}_{\infty}(J^*;U^*e_i)=E_s$ and $P(\lambda_s^*;J^*)U^{-1}b$ is a nonzero vector in $ E_s$. According to part $(c)$ of Theorem \ref{dyn-sampling}, $\lambda_s$ is a root of $q_{S_{\Omega},b}^A$. Therefore, all eigenvalues of $A$ can be recovered. 
 
 $(1)\implies (2):$ if there exists a $s\in [n]$, such that $\sum_{i\in \Omega}P(\lambda_s^*;J^*)\mathcal{K}_{\infty}(J^*;U^*e_i)$ is a proper subspace of $E_s$.  Let the nonzero vector $f_s \in E_s$ be orthogonal to  $\sum_{i\in \Omega}P(\lambda_s^*;J^*)\mathcal{K}_{\infty}(J^*;U^*e_i)$. Then for any $b\in \mathcal{S}$, the vector $\tilde{b}=b-UP(\lambda_s;J)b+Uf_s \in \mathcal{S}$, and we have
 $$\langle P^*(\lambda_s;A)e_i,  A^t\tilde b \rangle=0 \text{ for  all } i \in \Omega, $$
 which yields a contradiction for the universality of $\Omega$. 
 
 Now we prove the equivalence of $(2)$ and $(3)$. We use Theorem \ref{tadcul} in the appendix, which is a slight modification of the theorem 2.6 in \cite{aldroubi2017dynamical}. 
\end{proof}

Part (2) of Theorem \ref{universal} relates the universal construction of $\Omega$ to the trajectory behaviour of a conjugate dynamical system to \eqref{HDS}: consider trajectories starting with initial conditions  $\{e_i: i\in \Omega\}$, they should not belong to any proper subspace.  Part (3) of Theorem \ref{universal} indicates that the universal construction only depends on the invertibility of certain submatrices of $U$. It guides us to construct universal sampling set according to the spectral structure of  $A$.  For simplicity, we  present an example for a diagonalizable matrix. 

 If  $A=UDU^{-1}$ where $D$ is a  diagonal matrix as in \eqref{D1}, then the part (3) can be simplified as 
$\{P(\lambda_s;D)U^*e_i,  i \in \Omega\} $ spans  $\mathrm{Range}( P(\lambda_s;D))=E_s$. That is to say, the submatrix of $U^*$ formed by selecting column indices according to $\Omega$ and row indices according to $E_s$ needs to have full row rank. If we know $\mathrm{dim}(E_s)\equiv 1$ for $s=1,\cdots,d$, then it suffices to finding a column of $U^*$ which is nonzero everywhere.  This example can be immediate to generalize to the Jordan case where there is only one cyclic vector for each nilpotent block. Below, we present a more complicated example.

%This theorem indicate that the universal set $\Omega$ only depends on spanning property of the  eigen-basis of $A$ on the support given by $\Omega$. 

\begin{example}\label{circlegraph}[Diffusion process over circulant graphs.] In this system, the dynamics is generated by a circulant matrix $A$.    It is well-known that $A$ admits the spectral decomposition $A=UDU^{-1}$, where $U$ is the discrete Fourier matrix and  $D$ is a diagonal matrix, with eigenvalues $\lambda_1,\cdots, \lambda_n$. Note that in this case, $D$  is similar to the standard diagonal form in \eqref{D1} up to  a  permutation of the diagonal entries.  Let $E_s:=\mathrm{Range}(P(\lambda_s;D))$ and $s_{max}:=\max_{s=1,\cdots,n}{\mathrm{dim}(E_s)}$. Let us choose $\Omega=\{1,\cdots, s_{max}\}$,  and denote by $(U^*)_{\Omega}$ the submatrix of $U^*$  formed by choosing columns according to $\Omega$.  Then the matrix $(U^*)_{\Omega}$ satisfies the full spark property:  any $s_{max}$ rows of  $(U^*)_{\Omega}$  will form a Vandermonde matrix with distinct nodes, and therefore is invertible.  As a result, $\{P(\lambda_s;D)U^*e_i,  i \in \Omega\} $ spans  $E_s$ for $s=1\cdots,n$ and $\Omega$ is a universal construction. 

\end{example}

Example \ref{circlegraph} reveals a connection of universal sampling sets with full-spark submatrices of $U^*$, and the latter problem  has been extensively studied in the field of compressed sensing. For discrete Fourier matrix, Chebotar$\ddot{e}$v  showed that every square submatrix of  the discrete Fourier matrices is invertible if the dimension $d$ is prime \cite{stevenhagen1996chebotarev}. In this case,  an arbitrary set $\Omega \subset [d]$ with $|\Omega|=s_{\max}$ is universal.  In general, the deterministic construction is very difficult. As is done in compressed sensing, one can look for the random constructions of $\Omega$  at uniform such that any $s_{max}$ rows of  $(U^*)_{\Omega}$  satisfies the restricted isometry property.  It is shown that the cardinality of $\Omega$ which will be slightly more than $s_{\max}$, with proportional factor mainly depending on the coherence of the columns $U^*$, a quantity measuring the dependence.  We refer the readers \cite{candes2006robust,tang2017universal} for more details.

 \subsubsection{Case $c\neq 0$.}
In this subsection, we deal with the  affine systems with $c\neq 0$. Inspired by the idea in \cite{duan2020identification},  we show that the affine  systems can be transformed to a homogeneous system by  a linear transformation.  In this way, we are able to extend the previous results to the affine systems. For the sake of the conciseness, we only state the generalization of Theorem \ref{dyn-sampling}. Other results in section \ref{homogeneous} can be derived similarly. 

Now let us consider  the affine  system: 
$$x_{t+1}=Ax_{t}+c, x_0=b, t=0,1,\cdots,$$

Using the solution formula \eqref{sol1}, we have $x_{t+1}-x_t=A^t((A-I)b+c)$. Now we denote $y_t=x_{t+1}-x_t$, then we have
\begin{align}
y_t&= A^ty_0, t=0,1,2,\cdots \label{homo}\\
y_0&=(A-I)b+c; 
\end{align}
  
  Now we apply Theorem \ref{dyn-sampling} to the homogeneous system \eqref{homo}, we obtain the following theorem:
  \begin{theorem}\label{theorem0}
     Assume that the evolution matrix $A \in \mathbb{C}^{d\times d}$ and its Jordan decomposition can be written as $A=UJU^{-1}$ where $J$ is a Jordan matrix as in \eqref{Jordan1}. Let $b \in \mathbb{C}^d$ be the initial state and $c$ be the force term of the system. Define the vector $w=(A-I)b+c$.  Then  the polynomial $q_{S_{\Omega},w}^A$ can be  uniquely determined from $\{S_{\Omega}x_t: t=0,1,\cdots\}$. In addition, the following statements are equivalent: for each $s \in [n]$
\begin{itemize}
\item[(1)] $\lambda_s$ is \textbf{not} a root of $q_{S_{\Omega},w}^A$. 
\item[(2)]$w_s\perp \mathcal{K}_{\infty}(N_s^{\ast},\hat u_{s,i})$ for all $i\in \Omega$ where $w_s=P(\lambda_s; J)U^{-1}w$, $\hat u_{s,i} = P(\lambda_s; J)U^*e_i$. 
\item[(3)] $P^*(\lambda_s;A)e_i \perp \mathcal{K}_{\infty}(A;w) $ for all $i\in \Omega$. 
%\item[(d)] $b \not\perp P(\lambda_s^*; A^*)\mathrm{span}\{e_i,A^{*}e_i,\cdots (A^{*})^{r_{e_i}^A-1}e_i: i\in \Omega\} $

\end{itemize}
\end{theorem}      

Now we compare Theorem \ref{theorem0} with Theorem \ref{dyn-sampling}. The only difference is to replace the vector $b$ in Theorem 1 with  $w$.  Note that $w=(A-I)b+c$, so the nonzero external vector $c$ change the geometry of vectors and Krylov spaces. As a result, for  $\Omega\subset [d]$, $\lambda_s$ could  be recovered in the affine system but  may not in its homogeneous counterpart. In particular, when $A$ is the identity matrix, the initial condition $b$ plays no role in determining the recoverability of the eigenvalues in affine systems. This can also be seen from the solutions to the systems:  the initial condition $b$ just add a simple translation on the solutions. However, universal sampling sets for both affine systems and its homogeneous counterpart are the same, since it only depends on the spectral matrix $U$ by  Theorem \ref{universal}.
\subsection{Extension to continuous-time affine dynamical systems} 
In this section, we consider the continuous-time systems:
\begin{align*}
\dot x(t) &=Ax(t)+c, t\geq 0 \\
x(0)&=b;
\end{align*}

Suppose that  observation time instances are uniform, without loss of generality, say at $t=0,1, 2,\cdots$, we define 
$B=e^{A}$, and introduce  $y_t=x^{\mathrm{cont}}_{t+1}-x^{\mathrm{cont}}_t$, then we have \begin{align}\label{continuous}y_t&=B^ty_0\\ y_0&=(B-I)b+g(1;A)c,
\end{align}
 where we used the properties
 \begin{align}
g(t;A)&=tg( 1;tA),\\
(t+1)g(1;(t+1)A)-tg(1;tA)&= e^{tA}g(1;tA).
 \end{align}
 
 Therefore, we can apply the results of discrete-time systems to system \eqref{continuous}  generated by the evolution matrix $B$.  In general, the mapping $A\rightarrow e^A$ does not preserve the Jordan decomposition of $A$ except for the case when $A$ is similar to a diagonal matrix. For easy presentation, we state the  result for diagonalizable case.

   \begin{theorem}\label{theoremcont}
     Assume that the evolution matrix $A=UDU^{-1}$ where $D$ is a diagonal matrix. Let $b \in \mathbb{C}^d$ be the initial state and $c$ be force term of the system. Define the matrix $B=e^A$ and the vector $w=(B-I)b+g(1;A)c$.  Then  the polynomial $q_{S_{\Omega},w}^B$ can be  uniquely determined from $\{S_{\Omega}x^{\mathrm{cont}}_t: t=0,1,\cdots\}$. In addition, the following statements are equivalent: for each $s\in [n]$
\begin{itemize}
\item[(a)] $e^{\lambda_s}$ is \textbf{not} a root of $q_{S_{\Omega},w}^B$. 
\item[(b)]$w_s\perp\mathrm{span}\{\hat u_{s,i},i\in \Omega\}$,  where the vector $$w_s=\begin{cases}P(\lambda_s; D)U^{-1}(b+\frac{1}{\lambda_s}c), \text{ if } \lambda_s \neq 0 \\P(\lambda_s; D)U^{-1}c, \text{otherwise.}  \end{cases}$$ and the vector $\hat u_{s,i} = P(\lambda_s; D)U^*e_i$. 
\item[(c)] $P^*(\lambda_s;B)e_i \perp \mathcal{K}_{\infty}(B;w)$ for all $i\in \Omega$. 
%\item[(d)] $b \not\perp P(\lambda_s^*; A^*)\mathrm{span}\{e_i,A^{*}e_i,\cdots (A^{*})^{r_{e_i}^A-1}e_i: i\in \Omega\} $

\end{itemize}

    \end{theorem}  
    
  Furthermore, if $A$ has a real eigenvalue $\lambda$, then one can recover $\lambda$ from $e^{\lambda}$. However, this is not true for complex eigenvalues, 
    since $e^z$ is not a injective mapping on the complex domain $\mathbb{C}$.

\section{Numerical Algorithms}\label{sec:NAlg}
In this section, we propose various algorithms to recover the eigenvalues of $A$ from data. Recall that Proposition \ref{Prony} in fact provides us with a Prony-type algorithm to reconstruct the annihilating polynomial $q_{S_{\Omega},b}^A$. The original Prony method  was proposed  centuries ago to recover a vector with an $s$-sparse Fourier transform from $2s$ of its consecutive components. In recent years, the Prony method has been widely applied to different inverse problems including  the approximation of Green functions in fluid dynamics \cite{beylkin2007multiresolution}, the inverse scattering problem \cite{hanke2012one},  the parameter estimation in signal processing \cite{bossmann2012sparse}. The renaissance of Prony method originates from its stabilized variants , such as the ESPRIT method \cite{roy1986esprit}, the matrix pencil method \cite{hua1990matrix} or the approximate Prony method \cite{potts2010parameter,Prony-TLS1,Prony-TLS2}.  Error
estimates for the performance of Prony-type methods with noisy measurements  can be found in \cite{batenkov2013accuracy,potts2011nonlinear,filbir2012problem,tang2017system,beinert2021phase}. The robustness depends on the minimal gap between parameters.

To apply Prony-type methods, we first construct Hankel-like matrix. Let $\Omega \subset [d]$ and given the trajectory data $\{S_{\Omega}A^tb, t=0,\cdots, M-1\}$, the rectangular Hankel-like matrices for some positive integer $L<M$ is given by
\begin{equation}\label{bighankel}
\begin{split}
    &{H}_{\Omega, M-L,L+1}=\begin{bmatrix}\vline&\vline&&\vline\\ {h}_0&{h}_1&\cdots&{h}_L\\ \vline&\vline&&\vline\end{bmatrix},\\
    &{H}_{\Omega, M-L,L}(t) = H_{\Omega, M-L,L+1}(:,t+1:t+L),t=0,1,
\end{split}
\end{equation}
where the  $(M-L)|\Omega|\times 1$ column vector is
\begin{equation}\label{colhankel}
    {h}_l = \big[S_{\Omega}A^{l}b, S_{\Omega}A^{l+1}b,\cdots, S_{\Omega}A^{M-L+l-1}b\big]^T,
\end{equation} for $l=0,1,\cdots,L$.  Denote by $r=r_{S_{\Omega},b}^A$, the degree of minimal $(S_{\Omega},A,b)$ annihilator $q_{S_{\Omega},b}^A$. As long as  $M \geq L+ r$ and $L\geq r$, we have that  $\mathrm{rank}({H}_{M-L,L})=r$ and we can recover all $\Omega$-recoverable eigenvalues. A practical choice of $L$ is $L\lesssim M/2$ (and $L\leq N/2$), which is the maximum value for $L$.  In particular, when $L=r$, the Hankel-like matrix is full column rank.

%From Proposition \ref{Prony}, we see that it is important to determine $r$, which is also the number of recoverable eigenvalues (counting multiplicities),  to achive the best spectrum recovery performance. 

\paragraph{Determining the number of $\Omega$-recoverable eigenvalues.}  In applications, the degree $r$  of the annihilator $q_{S_{\Omega},b}^A$ is sometimes not known as a \textit{priori}. The information of $r$ is crucial in applying the Prony-typed method proposed in Proposition \ref{Prony}.  

\begin{alg}\label{getr}Determining the number of $\Omega$-recoverable eigenvalues.\\
    \textbf{Input:} $M\in\mathbb{N}$ s.t. $M>3d$, observed samples $\{S_{\Omega}A^tb, t=0,\cdots, M-1\}$.\\
    \indent \textbf{Step 1. Construct Hankel like matrix.} Construct $H_{\Omega, M-d, d+1}$ from \eqref{bighankel} and compute its singular values in descending order by taking SVD as $H_{\Omega, M-d, d+1}=U_{\Omega, M-d}D_{\Omega, M-d,d+1}W_{\Omega, d+1}^*$. \\
    \indent \textbf{Step 2. Detemine the numerical rank.} We have different methods to determine the numerical rank including: (i) Count the maximum number of singular values $r_1$ such that $\sigma_i\in\mathbb{R}~,i=1,2,\dots,r$ are greater than a given positive threshold $\epsilon$. (ii) Compute the quotient of singular values as $q_r = \frac{\sigma_r}{\sigma_{r+1}}$ and find the maximum quotient index $r_2=\max\limits_rq_r$. (iii) Reorder the quotient $q_r$ in descending order as $\hat q_r$, find the maximum gap $\hat q_r-\hat q_{r+1}$, and determine the original order of $\hat q_r$ in quotients $\{q_r\}$ as $r_3$.\\
    \textbf{Output:} The number of recoverable eigenvalues  $r_1$, $ r_2$, $ r_3$ produced by three criterions.
\end{alg}
We can estimate $r$ by calculating the numerical rank of the Hankel matrix  $H_{\Omega,M-L,L+1}$, given sufficient large $L$ and $M$. For example,  $L=d$ and $M\geq 2d$.  We  then compute  the  SVD (Singular Value Decomposition) of  $H_{\Omega,M-L,L+1}$ and  find $r$ as the number of singular values being larger than a threshold $\epsilon$, which could be a fixed constant or the largest spectral gap among singular values, see Algorithm 1 for more details.

%\begin{equation}\label{eq:SVD}
%    H_{M-L,L+1} \mathop{=}\limits^{def} U_{|S_\Omega|\times (M-L)} D_{|S_\Omega|\times (M-L),L+1} W_{L+1}^*.
%\end{equation}

%We then  inspect  its singular values $\tilde \sigma_1 \geq \tilde \sigma_2 \geq \dots \geq \tilde \sigma_{L+1} \geq 0$. We find $r$ as the number of singular values being larger than a predefined bound $\epsilon$. Usually, we can find a gap between $\tilde \sigma_M$ and the further singular values $\tilde \sigma_{M+1}, \dots, \tilde \sigma_{L+1}$, which are close to zero.  

%To begin with, we recall the rectangular Hankel matrix in \eqref{eq:Hankel} with $L\simeq \frac{1}{3} M$. Note that the amount of computation required  depends on the free parameter $L$. Numerical experiments show that the choice of $L$  greatly affects the noise sensitivity of the eigenvalues. In terms of the noise sensitivity and computation cost, the best choice for $L$ is between $\frac{N}{3}$ and $\frac{2N}{3}$ \citep{Hua1}. 

In the following subsections,  we will further investigate the data structure of the Hankel-like matrix ${H}_{\Omega,M-L,L+1}$ and present various  algorithms based on the ideas from classical Prony, matrix pencil method and ESPRIT estimation method. The latter two classical methods are well known for their better numerical stability than the original Prony method. % corollary 1.1.2

\subsection{The Prony LS and TLS Algorithm}
The Prony-type method contained in the proof of    Proposition \ref{Prony} is summarized as follows.
% First, we summarize the prony LS and TLS algorithm 
% %\citep{xxx}
%  for the identification of the linear evolution system, according to .

\begin{alg}Prony LS and TLS Method.\\
    \textbf{Input:} $M\in\mathbb{N}$ s.t. $M>3d$, observed samples $\{S_{\Omega}A^tb, t=0,\cdots, M-1\}$ and the number of recoverable eigenvalues $r$.\\
    % \indent \textbf{Step 1. Determining the number of recoverable eigenvalues.} Construct $H_{\Omega, M-d, d+1}$ from \eqref{bighankel} and compute its SVD as $H_{\Omega, M-d, d+1}=U_{\Omega, M-d}D_{\Omega, M-d,d+1}W_{\Omega, d+1}^*$. Determine its numerical rank $r$ by counting the number of singular values $\sigma_i\in\mathbb{R}~,i=1,2,\dots,d,d+1$ greater than a given positive threshold $\epsilon$ .\\
    \indent \textbf{Step 1. Construct and solve Hankel-like matrix equation.} Construct $H_{\Omega, M-r, r+1}$ from \eqref{bighankel} and compute its SVD as $H_{\Omega, M-r, r+1}=U_{\Omega,M-r}D_{\Omega,M-r,r+1}W_{\Omega,r+1}^*$. Construct the linear equation in \eqref{eq:Prony}: $H_{\Omega,M-r,r}(0)\vec{q}=-h_{\Omega, r}$ from \eqref{bighankel} and \eqref{colhankel} and compute its Least Square solution and Total Least Square solution $\vec{q} = (q_0,q_1,\dots,q_{r-1})^\top$ as
    \begin{align*}
        \vec{q}^{~LS}&= -H_{\Omega,M-r,r}(0)^+ h_{\Omega, r+1},\\
        \vec{q}^{~TLS}&= W_{\Omega,r+1}(1:r,r+1)/W_{\Omega,r+1}(r+1,r+1).
    \end{align*}
    \vspace{-0.2in}
    \\
    \indent \textbf{Step 2. Find the roots of the polynomial.} Compute all zeros $z_j^{LS},~z_j^{TLS}\in\mathbb{C}~,j=1,2,\dots,r$ of $q(z)=\sum_{k=0}^{r-1}q_k z^k + z^r$ with $\vec{q}^{~LS}$ and $\vec{q}^{~TLS}$, respectively.\\
    \textbf{Output:} $\{z_j^{~LS}\in\mathbb{C},~j=1,2,\dots,r\}$ and $\{z_j^{~TLS}\in\mathbb{C},j=1,2,\dots,r\}$.
\end{alg}

\subsection{Generalized Matrix Pencil Method}

In this section, we show that one can generalize the idea of Matrix pencil to recover the eigenvalues of $A$.  We first present a decomposition  of  Hankel matrices ${H}_{\Omega,M-L,L}(t), t=0,1$ for the case when $\Omega$ has a single point. 

\begin{lemma} \label{mtxdecomp}Let $\Omega = \{i\} \subset [d]$. Denote by $r=r_{S_{\Omega},b}^A$ and $r_s=r_{S_{\Omega},b_s}^A$ for $s=1,\cdots,n$. Then the rectangular Hankel matrices can be factorized in the following form for $t=0,1$ : 
\begin{equation}\label{factorization2}
    {H}_{M-L,L}(t) = {V}^T_{r,M-L} {\Lambda} {\hat J}^{t} {V}_{r,L}
\end{equation}\par

\begin{equation*}
\begin{split}
   \quad {V}_{r,L} :=& 
 \begin{bmatrix}
{V}_{r_1,L}\\
 {V}_{r_2,L}\\
   \vdots\\
   \mathrm{V}_{r_n,L}
   \end{bmatrix}, V_{r_s, L}=  
  \begin{bmatrix}
   C_0^0&C_1^0\lambda_s&\dots&C_{r_s-1}^0\lambda_s^{r_s-1}&\dots&C_{L-1}^0\lambda_s^{L-1}\\
   0&C_1^1&\dots&C_{r_s-1}^1\lambda_s^{r_s-2}&\dots&C_{L-1}^1\ \lambda_s^{L-2}\\
   \vdots&\vdots&\ddots&\vdots&\ddots&\vdots\\
   0&0&\dots&C_{r_s-1}^{r_s-1} \lambda_s^0 &\dots&C_{L-1}^{r_s-1}\lambda_s^{L-r_s}\\
   \end{bmatrix}\in\mathbb{C}^{r_s\times L},
\end{split}
\end{equation*}
and a $r$-by-$r$ matrix $\Lambda$:
\begin{equation*}%\label{factorization}
    {\Lambda} :=
  \begin{bmatrix}
   \Lambda_1& 0&\cdots&0\\
   0 & \Lambda_2 &\cdots&0\\
   \vdots&\vdots&\ddots&\vdots\\
 0 & 0 & \cdots&\Lambda_n
   \end{bmatrix},
  \Lambda_s = \begin{bmatrix}\label{lambdas}
   \langle (U^{-1}b)_s, (U^*e_i)_s\rangle&\cdots&\langle N_s^{r_s-1}(U^{-1}b)_s,(U^*e_i)_s \rangle\\
   \langle N_s(U^{-1}b)_s,(U^*e_i)_s\rangle&\cdots&0\\
   \vdots&\vdots&\vdots\\
\langle N_s^{r_s-1}(U^{-1}b)_s,(U^*e_i)_s\rangle&\cdots&0 \end{bmatrix} \in \mathbb{C}^{r_s\times r_s},\\
\end{equation*}
and a Jordan matrix $\hat J$:
\begin{equation*}%\label{factorization2}
   \hat{J}:=
   \begin{bmatrix}   \lambda_1 + \hat N_1&0&\cdots&0\\
    0 &\lambda_2 + \hat N_2&\cdots&0\\
   \vdots&\vdots&\ddots&\vdots\\
  0&0&\dots&\lambda_n+\hat N_n
   \end{bmatrix},\hat N_s=
   \begin{bmatrix}
   0&0&\dots&0&0\\
   1&0&\dots&0&0\\
   0&1&\ddots&0&0\\
   \vdots&\vdots&\ddots&\vdots&\vdots\\
   0&0&\dots&1&0
   \end{bmatrix} \in \mathbb{C}^{r_s\times r_s}.\\
\end{equation*}

\end{lemma}

\begin{proof} We prove this lemma by the matrices in \eqref{factorization2} are entrywise identical. 
The $(m,l)$ th entry of the Hankel matrix ${H}_{\Omega,M-L,L}(0)$ is given by 
\begin{align*}\label{binomial}
     \langle A^{m+l-2}b, e_i\rangle =   \langle J^{m+l-2} (U^{-1}b), U^*e_i\rangle &=\sum_{s=1}^{n} \langle J_s^{m+l-2}(U^{-1}b)_s, (U^*e_i)s\rangle\\
        &= \sum \limits_{s=1}^n \sum \limits_{k_s = 0}^{r_s-1} {m+l-2 \choose{k_s}}\lambda_s^{m+l-2-k_s}\langle N_s^{k_s}(U^{-1}b)_s,(U^*e_i)_s\rangle.
 \end{align*}
 
 Then by  using  the identity 
 $${m+l-2 \choose{k_s}} =\sum_{i=0}^{k_s}{ m-1 \choose i}{l-1 \choose{k_s-i}},$$
 %Crm+n=C0m∗Crn+C1m∗Cr−1n+…+Crm∗C0n
 and comparing the coefficients for $\langle N_s^{i}(U^{-1}b)_s,(U^*e_i)_s\rangle$ for $i=0,\cdots,r_s-1$,  one can show that it is the same with the $(m,l)$ th entry of the matrix ${V}^T_{r,M-L} {\Lambda} {V}_{r,L}$.  Similarly, we can prove for the Hankel matrix ${H}_{\Omega,M-L,L}(1)$. 
\end{proof}

%In the decomposition \eqref{factorization2}, we have that $\mathrm{rank}({V}^T_{r,M-L})=\mathrm{rank}( {V}_{r,L})=r$.  If $r_s\geq 1$, then the matrix $\Lambda_s$ is invertible by applying the Proposition \ref{Prony} to the matrix $N_s$. Otherwise, $\Lambda_s$ is an empty matrix. 

Recall that the superscripts ``$*$'' and ``$+$" will denote the conjugate transpose and the pseudoinverse. The following lemma provides a theoretical foundation for the Generalized Matrix Pencil method.

\begin{proposition}
Let $\Omega = \{i\} \subset [d]$. Denote by $r=r_{S_{\Omega},b}^A$ and $r_s=r_{S_{\Omega},b_s}^A$ for $s=1,\cdots,n$. Without loss of generality, assume that $r_s\geq 1$. Let $M,L$ be two postive integers such that $ r \leq L \leq M-r$. The solutions to the generalized singular eigenvalue problem:
\begin{equation}\label{matrixpencil0}
(z{H}_{M-L,L}(0)-{H}_{M-L,L}(1))v=0
\end{equation}
subject to $b \in \mathrm{Col}({H^{*}}_{\Omega,M-L,L}(0))$,  denoting the column space of ${H^{*}}_{\Omega,M-L,L}(0)$, are
\[ z_s=\lambda_s\]
\[v= \text{ the $r_1+\cdots+r_s$ th column of }{V}_{r,L}^+\]
for $s=1,\cdots,n.$
\end{proposition}

\begin{proof} Using the factorization \eqref{factorization2}, we re-write the equation \eqref{matrixpencil0} as
$$  {V}^T_{r,M-L} {\Lambda} (z-{\hat J}) {V}_{r,L}v=0. $$

Since the matrix $V_{r,L}$ has linearly independent rows,  every $x$ in $\mathrm{Col}({H^{*}}_{\Omega,M-L,L}(0))$  can be represented as $v=V_{r,L}^{+}c$ for some vector $c \in \mathbb{C}^r$ due to the property of psuedo-inverse. We also note that the matrix ${V}^T_{r,M-L}$ has linearly independent columns, and the matrix $\Lambda$ is invertible (by Proposition \ref{Prony}). Therefore, it suffices to solving 
$$    (z-{\hat J}) c=0. $$ 
We then know the possible values for $z$ such that the linear equations have nonzero solution are $\lambda_1,\cdots,\lambda_n$. The conclusion follows by solving the corresponding linear system of equations.  
\end{proof}

\paragraph{The general case. }  Let $\Omega=\{i_1,\cdots, i_k\}\subset [d]$, and assume that $r_s=r_{S_{\Omega,b_s}^A} \geq 1$ for $s=1,\cdots,n$. By appropriate permutations of rows, the Hankel-like matrix $H_{\Omega, M-L,L+1}$ defined in  \eqref{bighankel} can be transformed as 
\begin{equation}\label{permutationhankel}
\tilde H_{\Omega, M-L,L+1}:=\begin{bmatrix}
H_{\{i_1\},M-L,L+1}\\
\vdots\\
H_{\{i_k\},M-L,L+1}
\end{bmatrix}.
\end{equation}
Then each Hankel matrix $H_{\{i_j\},M-L,L+1}$ can be still factorized in the form  of \eqref{factorization2} 
\begin{equation}\label{factorization21}
    {H}_{\{i_j\},M-L,L}(t) = {V}^T_{r,M-L} {\Lambda_{\{i_j\}}} {\hat J}^{t} {V}_{r,L}, t=0,1. 
\end{equation} In particular, if $r_{S_{\{i_j\},b_s}}^A=0$, then the corresponding $\Lambda_{{\{i_j\}},s}$ in \eqref{lambdas}  will be a $r_s \times r_s$ zero matrix. 

 The solutions to the generalized singular eigenvalue problem:
\begin{equation}\label{matrixpencil}
(z{H}_{\Omega,M-L,L}(0)-{H}_{\Omega,M-L,L}(1))v=0
\end{equation}
subject to $x \in \mathrm{Col}({H^{*}}_{\Omega, M-L,L}(0))$,  denoting the column space of ${H^{*}}_{\Omega,M-L,L}(0)$ are the same with the problem 

\begin{equation}\label{matrixpenciltilde}
(z\tilde{H}_{\Omega,M-L,L}(0)-\tilde{H}_{\Omega,M-L,L}(1))\tilde v=0
\end{equation}
subject to $\tilde v \in \mathrm{Col}({\tilde H^{*}}_{\Omega, M-L,L}(0))$, in the sense that $\tilde v$ is a permutation of $v$.  Due to the factorization \eqref{factorization2} and  that $r_s=\max\limits_{j=1,\cdots,k}r_{S_{\{i_k\}},b_s}^A$, we have that  $\mathrm{Col}({\tilde H^{*}}_{\Omega, M-L,L}(0))=\mathrm{Range}(V_{r,L}^{+})$. Therefore the solution to \eqref{matrixpenciltilde} is equivalent to 
\begin{equation}\label{simplehankel}
\Lambda_{\{i_j\}}(z-\hat J)\tilde c=0, j=1,\cdots,k.
\end{equation}

Since for each $s=1,\cdots,n$, there is at least one $\Lambda_{\{i_j\},s}$ in \eqref{lambdas} is invertible, therefore the  values of $z$ to solve \eqref{simplehankel} with nonzero $\tilde c$  are $\lambda_1,\cdots, \lambda_n$. We therefore obtain the following conclusion:

\begin{theorem}\label{mtx}
Let $\Omega\subset [d]$, and denote $r:=r_{S_{\Omega,b}^A}$ . Let $M,L$ be two postive integers such that  $ r\leq L \leq M- r $. The $L \times L$ matrix  ${H^+}_{\Omega, M-L,L}(0){H}_{\Omega,M-L,L}(1)$ has  the same eigenvalues with roots of  $q_{S_{\Omega},b}^A $ and  $L-r$ zeros as eigenvalues.
\end{theorem}

\begin{proof}
Left multiplying \eqref{matrixpencil} by ${H^+}_{\Omega,M-L,L}$, we have
\begin{equation}
{H^+}_{\Omega, M-L,L}(0){H}_{\Omega,M-L,L}(1)x=z {H^+}_{\Omega,M-L,L}(0){H}_{\Omega, M-L,L}(0)x.
\end{equation}
By property of the pseudoinverse, ${H^+}_{\Omega, M-L,L}(0){H}_{\Omega,M-L,L}(0)$ is the orthogonal projection onto  $\mathrm{Col}({H^{*}}_{\Omega,M-L,L}(0))$. Since $x \in \mathrm{Col}({H^{*}}_{\Omega,M-L,L}(0))$, it is easy to see that  $\lambda_1,\cdots, \lambda_n$ are $n$ eigenvalues of ${H^+}_{\Omega,M-L,L}(0){H}_{\Omega,M-L,L}(1).$  Since the rank of ${H^+}_{\Omega,M-L,L}(0){H}_{\Omega,M-L,L}(1)$ is $ r \leq L$, ${H^+}_{\Omega,M-L,L}(0){H}_{\Omega,M-L,L}(1)$ has $L- r_{S_{\Omega,b}^A}$ zero eigenvalues.  
\end{proof}

It is evident that one advantage of the matrix pencil method is the fact that there is no need to compute the coefficients of the minimal annihilating polynomial $q_{S_{\Omega},b}^A$. In this way, we  need  only solve a standard eigenvalue problem of a square matrix  ${H^+}_{\Omega,M-L,L}(0){H}_{\Omega,M-L,L}(1)$. In order to compute ${H^+}_{\Omega,M-L,L}(0){H}_{\Omega,M-L,L}(1)$.  Inspired by the idea of Algorithm 5 in \cite{MP-Hua2}, we can employ the Singular Value Decomposition(SVD) based Matrix Pencil Method for Hankel-like matrices.

\begin{proposition}
In addition to the conditions of Theorem \ref{mtx}, given the SVD of the Hankel-like matrix, 

$${H}_{\Omega,M-L,L+1}={U}_{\Omega, M-L}{\Sigma}_{\Omega, M-L,L+1} {W}^*_{\Omega, L+1},$$
then $${H^+}_{\Omega,M-L,L}(0){H}_{\Omega, M-L,L}(1)=\left({W}^*_{\Omega, L+1}(1:r,1:L)\right)^+\left({W}_{\Omega, L+1}^*(1:r,2:L+1)\right).$$
\end{proposition}

We now summarize the Generalized Matrix Pencil Method as below.

\begin{alg} Matrix Pencil Method \\
    \textbf{Input:} $M\in\mathbb{N}$ s.t. $M>3d$, observed samples $\{S_{\Omega}A^tb, t=0,\cdots, M-1\}$ and the number of recoverable eigenvalues $r$.\\
    % \indent \textbf{Step 1. Determining the number of recoverable eigenvalues.} Construct $H_{\Omega, M-L, L+1}$ with $L\simeq M/3$ from Equation \eqref{bighankel} and compute its SVD as $H_{\Omega, M-L, L+1}=U_{\Omega, M-L}D_{\Omega, M-L, L+1}W_{\Omega, L+1}^*$. Determine its numerical rank $r$ by counting the number of singular values $\sigma_i\in\mathbb{R}~,i=1,2,\dots,L,L\!+\!1$ greater than a given positive threshold $\epsilon>0$ .\\
    \indent \textbf{Step 1. Construct and Solve Hankel-like Matrix Equation.}Construct the matrix equation in as $H_{\Omega,M-L,L}(0)C=H_{\Omega,M-L,L}(1)$ from Equations \eqref{bighankel} and \eqref{colhankel} and compute SVD of augmented matrix $[H_{\Omega,M-L,L}(0) ~ H_{\Omega,M-L,L}(1)] = U_{\Omega, M-L}'D_{\Omega, M-L,2L}'W_{\Omega, 2L}^{\prime*}$ and compute its least square solution, total least square solution and SVD-based solution by
    \begin{align*}
        C^{~LS} &= \left(H_{\Omega,M-L,L}(0)\right)^+H_{\Omega,M-L,L}(1).\\
        C^{~TLS} &= \left(W_{\Omega, 2L}^{\prime} (1:L,L\!+\!1:2L)\right)^+W_{\Omega, 2L}'(L\!+\!1:2L, L\!+\!1:2L).\\
        C^{~SVD} &= \left(W^*_{\Omega, L+1} (1:r,1:L)\right)^+W^*_{\Omega, L\!+\!1}(1:r,2:L\!+\!1).
    \end{align*}
    \vspace{-0.2in}
    \\
    \indent \textbf{Step 2. Find the Eigenvalues of Companion Matrix.} Compute all eigenvalues $\lambda_j^{LS}$, $\lambda_j^{TLS}$, $\lambda_j^{~SVD}\in\mathbb{C}~,j=1,2,\dots,L$ of the companion matrix $C^{LS}~,C^{TLS}~,C^{SVD}$, repectively.\\
    \indent \textbf{Step 3. Remove redundant zeros.} Remove the eigenvalues with sufficiently small norm, i.e., $\|\lambda_j\|\leq \eta$ for a small positive number $\eta \in \mathbb{R}^+$.\\
    \textbf{Output:} $\{z_j^{~LS} \in \mathbb{C}, j=1,2,\dots,L ~ | ~ \|\lambda_j^{~LS}\|>\eta \}$, $\{z_j^{~TLS} \in \mathbb{C}, j=1,2,\dots,L ~ | ~ \|\lambda_j^{~TLS}\|>\eta \}$, $\{z_j^{~SVD} \in \mathbb{C}, j=1,2,\dots,L ~ | ~ \|\lambda_j^{~SVD}\|>\eta \}$.
    
\end{alg}

\subsection{Generalized ESPRIT Method}
%In this section, we will compare the performance among ESPRIT-like algorithm, including PMV-ESPRIT, TLS ESPRIT, MEPM \citep{MEPM}, Sliding-Window ESPRIT\ \citep{Sliding-Window-ESPRIT}.

The original ESPRIT Method relies on a particular property of  Vandermonde matrices known as  rotational invariance \cite{roy1986esprit}. By the factorization \eqref{factorization2}, and the permutation argument in \eqref{permutationhankel}, we have seen that the Hankel-like data matrix ${H}_{\Omega,M-L,L+1}$  is rank-deficient and that its range space, spanned by columns of $V_{\Omega,r,M-L}^T$, satisfies a generalized rotation invariance property:
\begin{align}\label{eq:rot-inv}V_{\Omega,r,M\!-\!L}^T(2:M\!-\!L,:)=V_{\Omega,r,M-L}^T(1:M\!-\!L\!-\!1,:)\hat J^T,\end{align}
where the matrix $V_{\Omega,r,M-L}$ and the Jordan matrix $\hat J$ is defined in \eqref{factorization2}.   Hence, we can generalize the ESPRIT algorithm based on SVD for estimating the eigenvalues of $\hat J$ in our setting. We summarize this method below: 

\begin{alg} ESPRIT Method \\
    \textbf{Input:} $M\in\mathbb{N}$ s.t. $M>3d$, observed samples $\{S_{\Omega}A^tb, t=0,\cdots, M-1\}$ and the number of recoverable eigenvalues $r$.\\
    % \indent \textbf{Step 1. Determining the Number of Recoverable Eigenvalues.} Construct $H_{\Omega, M-L, L+1}$ with $L\simeq M/3$ from Equation \eqref{bighankel} and compute its SVD as $H_{\Omega, M-L, L+1}=U_{\Omega, M-L}D_{\Omega, M-L,L+1}W_{\Omega, L+1}^*$. Determine its numerical rank $r$ by counting the number of singular values $\sigma_i\in\mathbb{R}~,i=1,2,\dots,L,L\!+\!1$ greater than a given positive threshold $\epsilon>0$ .\\
    \indent \textbf{Step 1. Construct and Solve Matrix Equation.} Construct the generalized rotation invariance matrix equation as in \eqref{eq:rot-inv} and compute its solution by
    \begin{align*}
        \hat J = \left(U_{\Omega,M\!-\!L} (1:M\!-\!L\!-\!1,1:r)\right)^+U_{\Omega,M-L}(2:M\!-\!L, 1:r).
    \end{align*}
    \vspace{-0.2in}
    \\
    \indent \textbf{Step 2. Find the Eigenvalues of Matrix.} Compute all eigenvalues $\lambda_j^{~ES}\in\mathbb{C}~,j=1,2,\dots,r$ of matrix $\hat J$.\\
    \textbf{Output:} $\{\lambda_j^{~ES} \in \mathbb{C}, j=1,2,\dots,r\}$.
\end{alg}

\section{Empirical evaluations}
In this section, we examine and compare the performance of Algorithm 1-4 on estimating spectrum of various affine systems (see Section \ref{sec:NAlg}). Our focus is the noise-free data and consider the system matrix $A$ whose operator norm is no greater than 1. Denote by $r$  the number of recoverable eigenvalues, and  we let $\Lambda = \{\lambda_1,\cdots, \lambda_r\}$ be $\Omega$-recovered eigenvalue set,  and $\hat \Lambda = \{\hat \lambda_1,\cdots, \hat\lambda_r\}$ be the eigenvalue set obtained by our numerical algorithms. The performance of the algorithms is measured by the root mean squared error (RMSE) and the infinity norm error (INE) defined as
\begin{align}
    \mathrm{RMSE}(\Omega, M,\sigma)&=\sqrt{\frac{1}{r} \sum_{i=1}^{r} (\lambda_i-\hat \lambda_i)^2},\\
    \mathrm{RINF}(\Omega, M,\sigma)&=\max_{i=1,\cdots r}|\lambda_i-\hat \lambda_i|.
\end{align}

In each example, we obtain the observational trajectory data of the form $\{S_{\Omega}x_t: t=0,1,\cdots, M-1\}$, where $\Omega\subset [d]$.  In numerical experiements, we get an estimation $\hat r$ of the number of recoverable eigenvalues using Algorithm \ref{getr} and then use the Hausdorff distance to match the recoverable eigenvalues with the exact ones as
\begin{equation}
\lambda_i = \mathop{\text{arg}\, \text{min}}\limits_{\lambda\in\Lambda} (\lambda-\hat\lambda_i)\text{, for }\hat\lambda_i \in \hat\Lambda.
\end{equation}

The results and details are listed below.

% \subsection{Toy example - a discrete affine system }

\paragraph{Example 1. Discrete Affine System.}  Recall that a discrete state-time affine system is given by 
$$x_{t+1}=Ax_t+c.$$ We consider a system of dimension 8. Denote by $J=\text{diag}(0.3I_3+\hat N_3, 0.5I_2+\hat N_2, 0.6,-0.2I_2)$, and $U=\text{diag}(I_3,~\text{toeplitz}([1,0,0],[1,1,1]),~\text{hankel}([1,2],[2,1]))$, where $I_s$ denotes the $s\times s$ identity matrix, and $\hat N_s$ denotes $s\times s$ nilpotent matrix with one cyclic block as in Lemma \ref{mtxdecomp}. We have that $A=UJU^{-1}$ with the initial condition $x_0=[8,7,\cdots,1]^T$ and $c= [1,1,\cdots,1]^T$. To illustrate the reconstruction flows, we first depict the observations of the discrete affine system for the choice of $\Omega=\{1,2,4,7\}$ in Figure \ref{fig:discrete-affine-system} (left). Note that by introducing $y_t=x_t+(A-I_8)^{-1}c$, the discrete affine system is reduced to the linear dynamical system. We can then apply Algorithm \ref{getr} to construct Hankel-type matrix $H_{\Omega, M-L, L}$ and compute its quotients of singular values $\sigma_i/\sigma_{i+1}$ for $i=1,2,\dots,8$ as shown in Figure  \ref{fig:discrete-affine-system} (right). In this case, 6 eigenvalues are $\Omega$-recoverable, matching the analytical results in Theorem \ref{dyn-sampling}. To further compare the performance of various algorithms and investigate the impact of $\Omega, M$, we conduct the reconstruction with various choices of $\Omega$ and $M$ as shown below in Table \ref{table:discrete-affine-system-Omega} and \ref{table:discrete-affine-system-M}. 

\begin{figure}[!htbp]
    \includegraphics[width=0.5\linewidth]{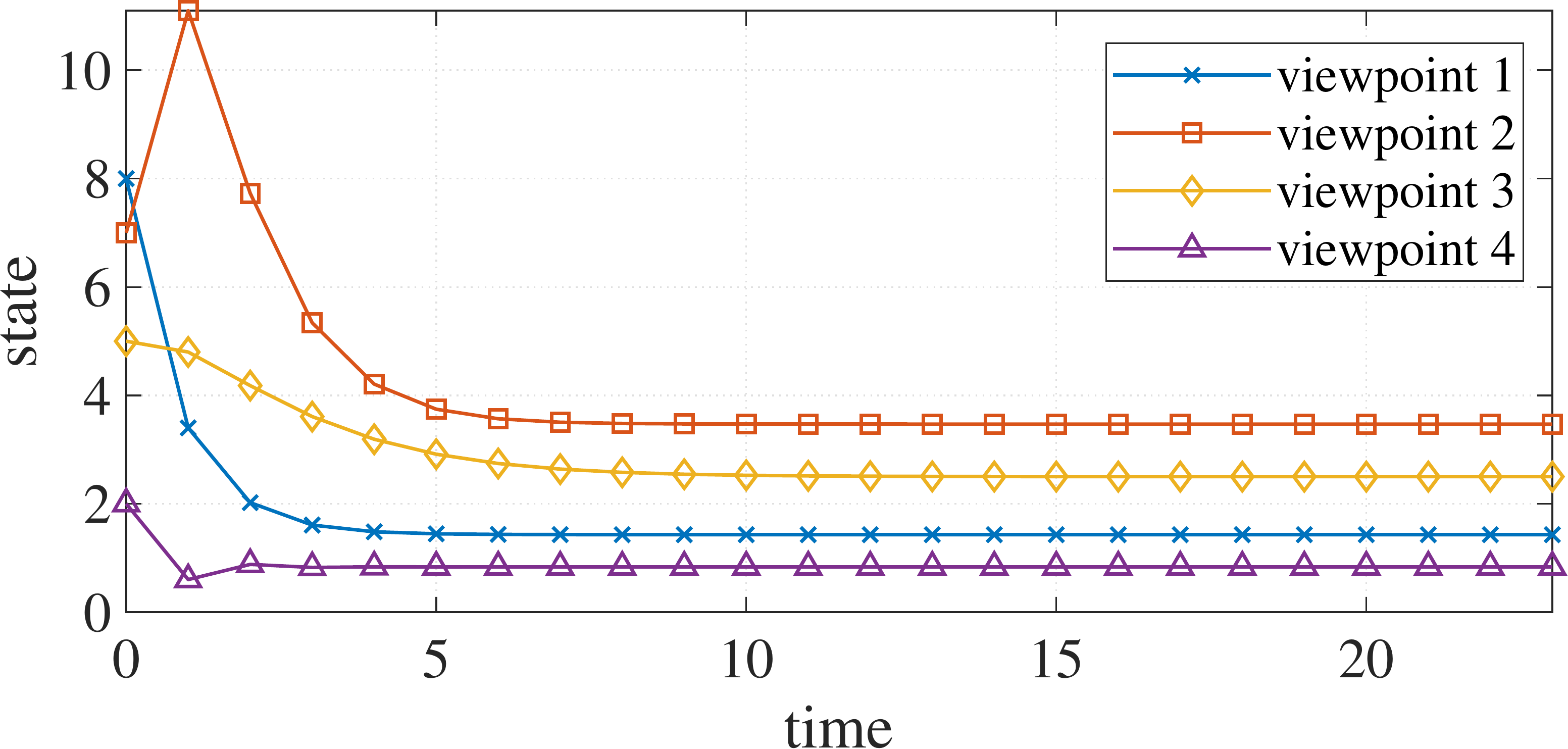}
    \includegraphics[width=0.5\linewidth]{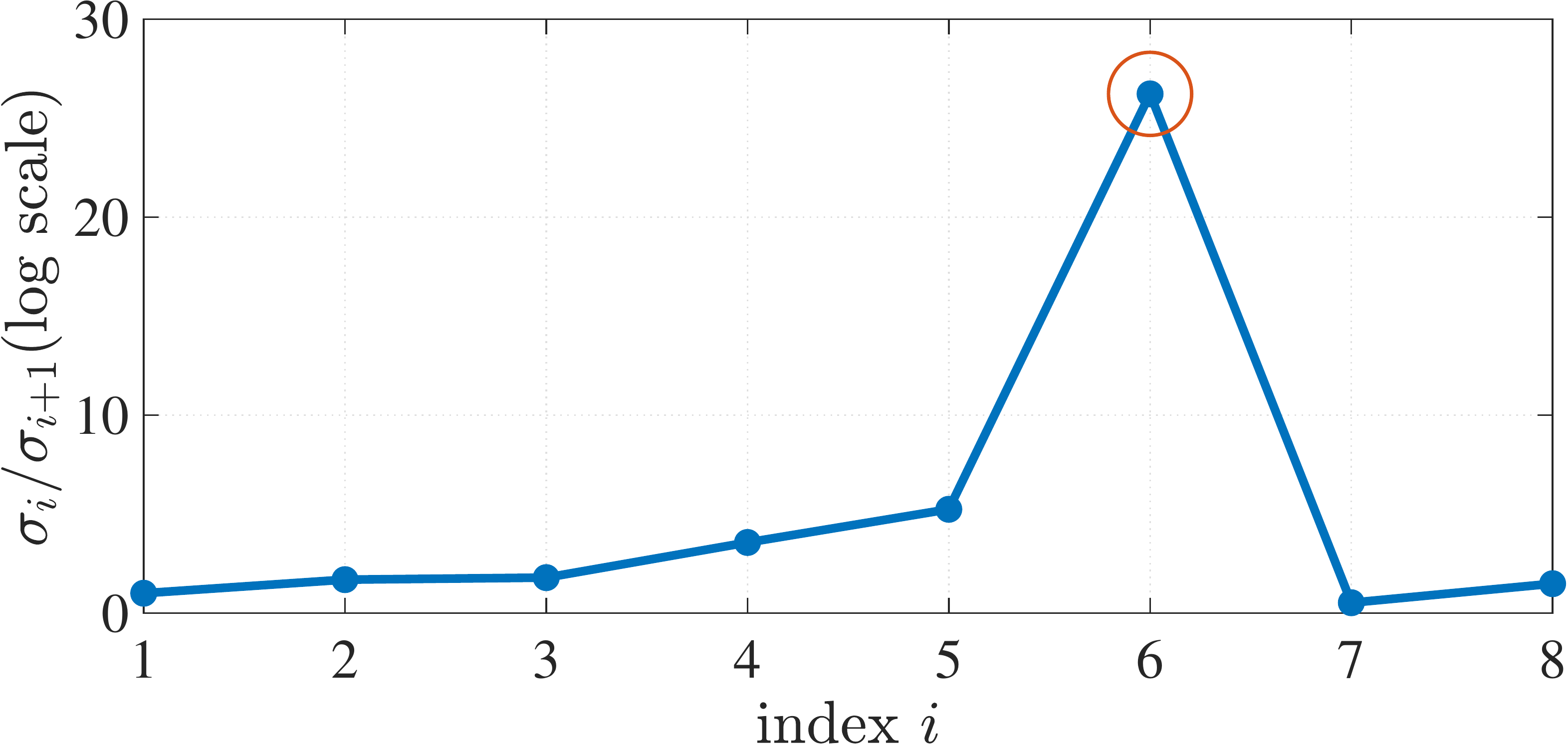}
    \caption{\textbf{Example 1.} (Left) State-time plot of observations of the trajectory with $\Omega=\{1,2,4,7\}$. (Right) Spectrum quotient of Hankel-type matrix for the choice of $\Omega=\{1,2,4,7\}$. }
    \label{fig:discrete-affine-system}
\end{figure}
\begin{table}[!h] 
    {\tiny
    \begin{center}
    \begin{tabular}
    {|p{4.3em}|p{1.2em}       |p{0.5em}|p{0.5em}     |p{5em}|p{5em}    |p{5em}|p{5em}          |p{5em}|p{5em}    |p{5em}|p{5em}    |}
    \hline
    \multirow{2}{*}{$\Omega$}     &\multirow{2}{*}{$M$}          &\multirow{2}{*}{$r$}          &\multirow{2}{*}{$\hat r$}     &\multicolumn{2}{l|}{Prony}                                &\multicolumn{2}{l|}{Matrix Pencil}                           &\multicolumn{2}{l|}{Matrix Pencil SVD}                       &\multicolumn{2}{l|}{ESPRIT}                                  \\ \cline{5-12}
                                  &                              &                              &                              &RMSE                          &INE                           &RMSE                          &INE                           &RMSE                          &INE                           &RMSE                          &INE                           \\ \hline
    $\{1\}$                       &24                            &1                             &1                             &$ 3.9 \cdot 10^{-17}$         &$ 1.1 \cdot 10^{-16}$         &$ 3.9 \cdot 10^{-17}$         &$ 1.1 \cdot 10^{-16}$         &$ 2.0 \cdot 10^{-17}$         &$ 5.6 \cdot 10^{-17}$         &$ 2.0 \cdot 10^{-17}$         &$ 5.6 \cdot 10^{-17}$         \\ \cline{5-12}
    $\{1,4\}$                     &24                            &4                             &4                             &$ 4.3 \cdot 10^{-8}$          &$ 8.7 \cdot 10^{-8}$          &$ 1.9 \cdot 10^{-8}$          &$ 3.8 \cdot 10^{-8}$          &$ 4.4 \cdot 10^{-8}$          &$ 8.7 \cdot 10^{-8}$          &$ 1.7 \cdot 10^{-8}$          &$ 3.4 \cdot 10^{-8}$          \\ \cline{5-12}
    $\{1,4,7\}$                   &24                            &5                             &5                             &$ 6.0 \cdot 10^{-8}$          &$ 1.2 \cdot 10^{-7}$          &$ 4.0 \cdot 10^{-8}$          &$ 8.0 \cdot 10^{-8}$          &$ 3.3 \cdot 10^{-8}$          &$ 6.5 \cdot 10^{-8}$          &$ 4.8 \cdot 10^{-4}$          &$ 9.7 \cdot 10^{-4}$          \\ \cline{5-12}
    $\{1,2,4,7\}$                 &24                            &6                             &6                             &$ 5.5 \cdot 10^{-8}$          &$ 9.9 \cdot 10^{-8}$          &$ 1.5 \cdot 10^{-8}$          &$ 2.9 \cdot 10^{-8}$          &$ 5.7 \cdot 10^{-8}$          &$ 9.5 \cdot 10^{-8}$          &$ 4.8 \cdot 10^{-4}$          &$ 9.7 \cdot 10^{-4}$          \\ \hline
    \end{tabular}
    \end{center}}
    \caption{\textbf{Example 1.} Numerical rank and errors for various algorithms and choices of $\Omega$.}

    \label{table:discrete-affine-system-Omega}
    \end{table}
\vspace{-0.3in}

\begin{table}[!h] 
        
    {\tiny
    \begin{center}
    \begin{tabular}
        {|p{4.3em}|p{1.2em}       |p{0.5em}|p{0.5em}     |p{4.3em}|p{4.3em}    |p{5em}|p{5em}          |p{5em}|p{5em}    |p{5em}|p{5em}    |}
        \hline
        \multirow{2}{*}{$\Omega$}     &\multirow{2}{*}{$M$}          &\multirow{2}{*}{$r$}          &\multirow{2}{*}{$\hat r$}     &\multicolumn{2}{l|}{Prony}                                &\multicolumn{2}{l|}{Matrix Pencil LS}                        &\multicolumn{2}{l|}{Matrix Pencil SVD}                       &\multicolumn{2}{l|}{ESPRIT}                                  \\ \cline{5-12}
                                      &                              &                              &                              &RMSE                          &INE                           &RMSE                          &INE                           &RMSE                          &INE                           &RMSE                          &INE                           \\ \hline
        $\{1,2,4,7\}$                 &24                            &6                             &6                             &$ 5.5 \cdot 10^{-8}$          &$ 9.9 \cdot 10^{-8}$          &$ 1.5 \cdot 10^{-8}$          &$ 2.9 \cdot 10^{-8}$          &$ 5.7 \cdot 10^{-8}$          &$ 9.5 \cdot 10^{-8}$          &$ 4.8 \cdot 10^{-4}$          &$ 9.7 \cdot 10^{-4}$          \\ \cline{5-12}
        $\{1,2,4,7\}$                 &32                            &6                             &6                             &$ 8.7 \cdot 10^{-8}$          &$ 1.5 \cdot 10^{-7}$          &$ 9.1 \cdot 10^{-8}$          &$ 1.7 \cdot 10^{-7}$          &$ 6.2 \cdot 10^{-8}$          &$ 1.2 \cdot 10^{-7}$          &$ 4.3 \cdot 10^{-6}$          &$ 8.7 \cdot 10^{-6}$          \\ \cline{5-12}
        $\{1,2,4,7\}$                 &40                            &6                             &6                             &$ 1.1 \cdot 10^{-7}$          &$ 2.2 \cdot 10^{-7}$          &$ 7.6 \cdot 10^{-8}$          &$ 1.5 \cdot 10^{-7}$          &$ 1.3 \cdot 10^{-7}$          &$ 2.5 \cdot 10^{-7}$          &$ 8.7 \cdot 10^{-8}$          &$ 1.7 \cdot 10^{-7}$          \\ \cline{5-12}
        $\{1,2,4,7\}$                 &48                            &6                             &6                             &$ 1.1 \cdot 10^{-7}$          &$ 2.0 \cdot 10^{-7}$          &$ 8.3 \cdot 10^{-8}$          &$ 1.6 \cdot 10^{-7}$          &$ 1.1 \cdot 10^{-7}$          &$ 2.1 \cdot 10^{-7}$          &$ 8.3 \cdot 10^{-8}$          &$ 1.6 \cdot 10^{-7}$          \\ \hline
    \end{tabular}
    
    \end{center}}
    \caption{\textbf{Example 1.} Errors for various algorithms and choices of $M$.}

    \label{table:discrete-affine-system-M}
    \end{table}

\medskip
\paragraph{Dynamical processes on graphs.} Graph learning arises in a wide range of  applications. 
We consider a weighted {graph} $\mathcal{G}=\mathcal{(V,E,W)}$ in which $\mathcal{V}=\{v_1,\cdots, v_d\}$ is  set of $d$ vertices and $ \mathcal{E} \subset\mathcal{ V\times V}$ is a set of edges.  The weighted adjacent matrix $\mathcal{W}$ is defined as 
\begin{equation}
\mathcal{W}(i,j)\triangleq 
\begin{cases}
   \alpha_{ij}&\mbox{if  the directed pair $(v_i, v_j) \in \mathcal{E}$}\\
       0 &\mbox{otherwise}
\end{cases}; \alpha_{ij} \in \mathbb{R}_{+}; \forall v_i,v_j \in \mathcal{V}.
\end{equation}

 The {degree} $\mathrm{deg}(v_i)$ of a vertex $v_i \in \mathcal{V}$ is defined as $\mathrm{deg}(v_i)=\sum_{j=1}^{n} \mathcal{W}(i,j)$.   In the following, we introduce important operators associated with the graph $\mathcal{G}$.

\begin{definition}The normalized diffusion operator of a graph $\mathcal{G}$ with the weighted adjacent matrix $\mathcal{W} \in \mathbb{R}^{d \times d}$ is  defined by 
$\mathcal{A} \triangleq  (\mathcal{D}^{-1})^{\frac{1}{2}}\mathcal{W}  (\mathcal{D}^{-1})^{\frac{1}{2}}$, where $\mathcal{D} :=\mathrm{diag}(\mathrm{deg}(v_i))_{v_i\in \mathcal{V}}$ and $\mathcal{D}^{-1}$ denote its psedoinverse.   The  normalized graph Laplacian operator  is $\mathcal{L} = I-\mathcal{A}$. The random walk transition matrix  $\mathcal{P}$ is defined by $ \mathcal{D}^{-1}\mathcal{W}$.
\end{definition}

Note that if the  vetex $v_i$ in a graph is isolated, then the degree matrix $\mathcal{D}(i,i)=0$. In this case, we use psuedo-inverse of $\mathcal{D}$ to calculate the transition matrix $\mathcal{P}$ and set $\mathcal{P}(i,i)=1$.

\paragraph{(a) Random walk over graphs.} A random walk on graph is a  dynamical process comprised of a series of random steps by moving to an adjacent vertex at each step:  if $v(t)$ represents the vertex of the random walk at the timestep $t$ then we moves to the next one $v({t+1})$ by picking one of its neighbours with probability, 
\begin{equation}
\mathbb{P}(v(t+1)|v(t)) = \begin{cases}\frac{1}{\mathrm{deg}(v(t))},\quad\text{if $(v(t),v(t+1))\in \mathcal E$,}\\
    0,\quad\text{otherwise,}
\end{cases}
\end{equation}
where $\mathrm{deg}(v(t))$ denotes the number of edges starting from $v(t)$ in digraph $\mathcal{G}$. Let $x_t$ denote the probability distribution  at time $t$ $$x_t(i) = \mathbb{P}(v(t)= v_i).$$

By rule of random walk, we have the following linear evolution system,
$$x_t =A^tx_0, A=\mathcal{P}^T.$$

The eigenvalues of $A$ reveals useful information about the underlying graph: the multiplicity of eigenvalue $1$ is equal to the number of  (strongly) connected components; the second largest eigenvalue $\lambda_2$ that describes the mixing rate of the random walks; the spectral gap $|\lambda_1-\lambda_2|$ represents how well the graph is connected. We refer the readers to the book \cite{chung1997spectral} for more connections between them.

% \begin{wrapfigure}{r}{0.4\linewidth}
%         \vspace{0.4in}
%         \hspace{0in}\includegraphics[width=6cm]{./figures/rand_walk_hankel_specgap.eps}
    
% \end{wrapfigure}
\paragraph{Example 2: a directed unweighted graph. } We consider a simple directed unweighted graph of 20 nodes. Its weighted adjacent matrix $\mathcal{W}$ (its nonzero entries are all 1s) is randomly generated with 80 edges and we then remove the self-loops. The initial state $x_0$ is a non-degenerate discrete probability distribution on $\{1,\cdots,20\}$. The matrix $A$ in our example is an invertible  diagonalizable matrix with 20 distinct eigenvalues and its eigen-matrix $U$ is entrywise nonzero. We summarize the reconstruction results below. In particular, we recovered the multiplicity of 1 in this case,  indicating that the graph has only 1 strongly connected component.

\begin{figure}[H]
    \includegraphics[width=0.4\linewidth]{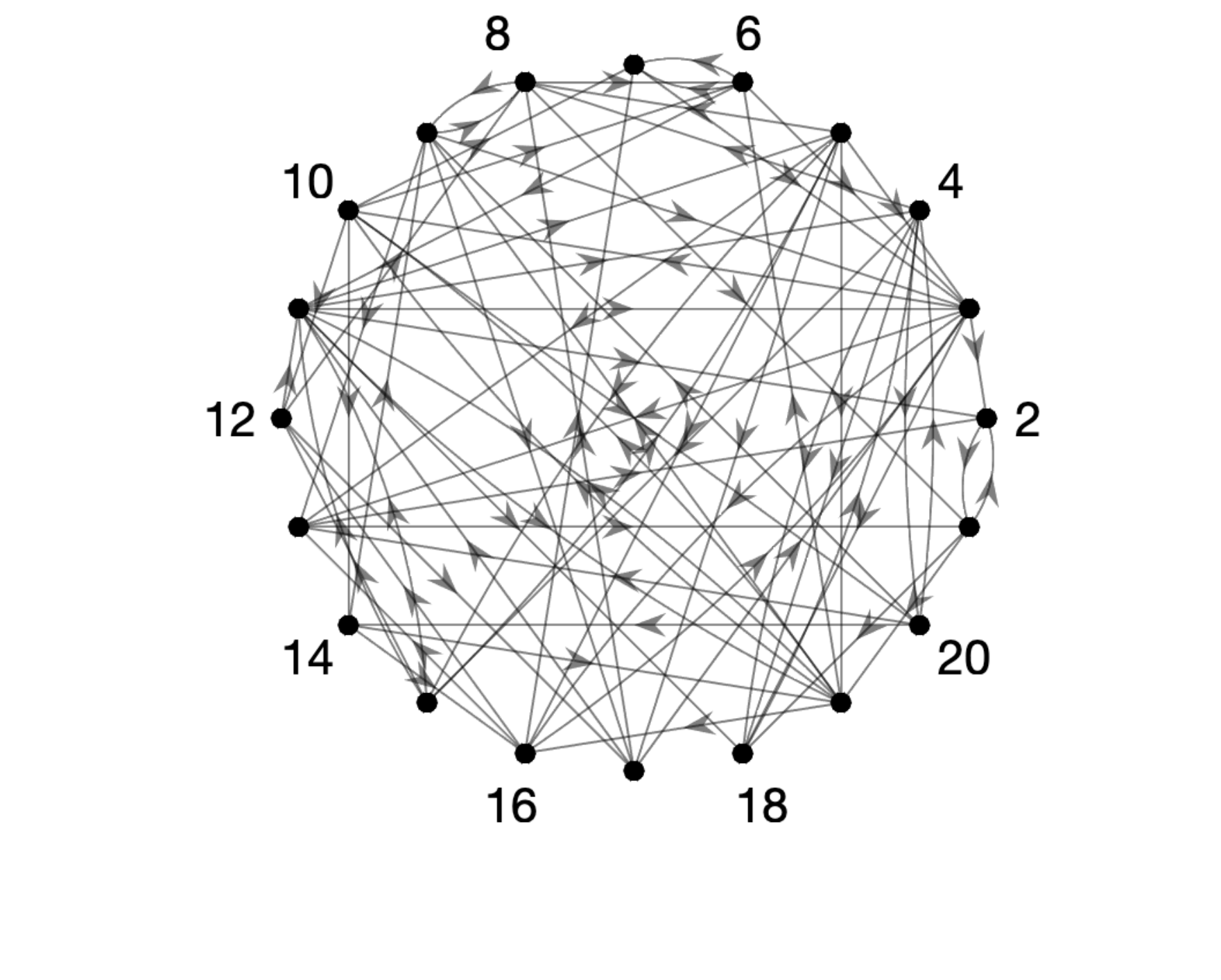}
    \includegraphics[width=0.6\linewidth]{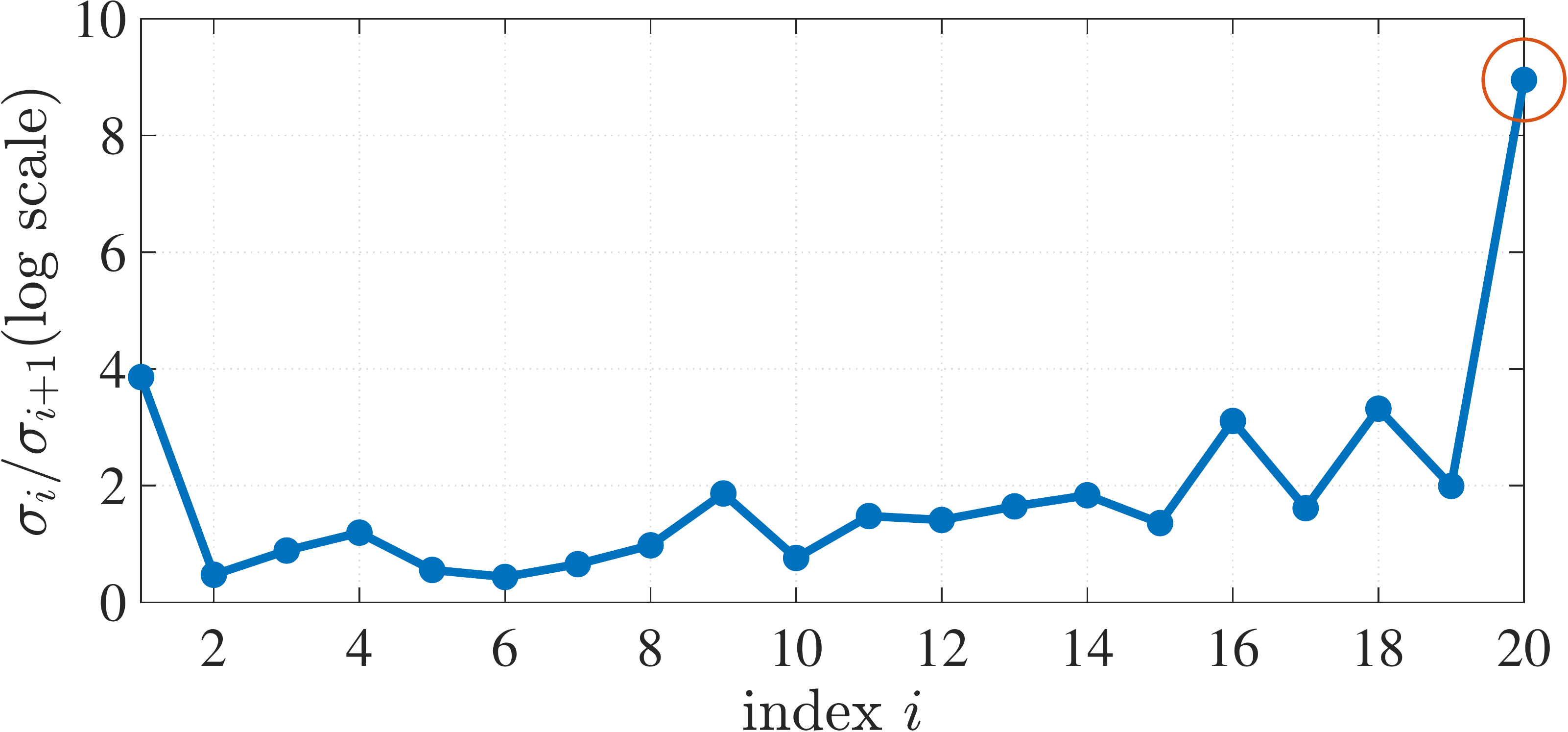}
    \caption{\textbf{Example 2.} (Left) An illustration of vertices and edges in the directed unweighted graph. (Right) Spectrum quotient of Hankel-type matrix for the choice of $\Omega=\{1,2,\dots,7\}$. In this case, 20 eigenvalues are $\Omega$-recoverable. }
\end{figure}

\begin{table}[!h] 
    {\tiny
    \begin{center}
    \begin{tabular}
    {|p{5em}|p{1.2em}       |p{0.5em}     |p{4.3em}|p{4.3em}    |p{5em}|p{5em}          |p{4.3em}|p{4.3em}    |p{4.3em}|p{4.3em}    |}
    \hline
    \multirow{2}{*}{$\Omega$}     &\multirow{2}{*}{$M$}          &\multirow{2}{*}{$\hat r$}     &\multicolumn{2}{l|}{Prony}                                   &\multicolumn{2}{l|}{Matrix Pencil}                           &\multicolumn{2}{l|}{Matrix Pencil SVD}                       &\multicolumn{2}{l|}{ESPRIT}                                  \\ \cline{4-11}
                                  &                              &                              &RMSE                          &INE                           &RMSE                          &INE                           &RMSE                          &INE                           &RMSE                          &INE                           \\ \hline
    $\{1\}$                       &60                            &10                            &$ 2.8 \cdot 10^{-2}$          &$ 8.1 \cdot 10^{-2}$          &$ 4.0 \cdot 10^{-5}$          &$ 1.3 \cdot 10^{-4}$          &$ 2.8 \cdot 10^{-2}$          &$ 8.1 \cdot 10^{-2}$          &$ 2.7 \cdot 10^{-2}$          &$ 7.8 \cdot 10^{-2}$          \\ \cline{4-11}
    $\{1,2\}$                     &60                            &19                            &$ 7.6 \cdot 10^{-3}$          &$ 3.3 \cdot 10^{-2}$          &$ 4.2 \cdot 10^{-5}$          &$ 1.8 \cdot 10^{-4}$          &$ 7.6 \cdot 10^{-3}$          &$ 3.3 \cdot 10^{-2}$          &$ 1.5 \cdot 10^{-1}$          &$ 3.1 \cdot 10^{-1}$          \\ \cline{4-11}
    $\{1,2,3\}$                   &60                            &20                            &$ 7.3 \cdot 10^{-7}$          &$ 2.0 \cdot 10^{-6}$          &$ 7.1 \cdot 10^{-6}$          &$ 3.1 \cdot 10^{-5}$          &$ 1.2 \cdot 10^{-6}$          &$ 5.1 \cdot 10^{-6}$          &$ 8.9 \cdot 10^{-2}$          &$ 1.8 \cdot 10^{-1}$          \\ \cline{4-11}
    $\{1,2,3,4\}$                 &60                            &20                            &$ 2.0 \cdot 10^{-7}$          &$ 7.7 \cdot 10^{-7}$          &$ 6.3 \cdot 10^{-7}$          &$ 2.7 \cdot 10^{-6}$          &$ 1.1 \cdot 10^{-7}$          &$ 3.9 \cdot 10^{-7}$          &$ 6.7 \cdot 10^{-2}$          &$ 1.3 \cdot 10^{-1}$          \\ \hline
    \end{tabular}
    \end{center}}
    \label{table:randwalk}
    \caption{\textbf{Example 2.} Numerical rank and errors of various algorithms and choices of $\Omega$.}
\end{table}

\begin{table}[!h] 
    {\tiny
    \begin{center}
    \begin{tabular}
    {|p{5em}|p{1.2em}       |p{0.5em}     |p{4.3em}|p{4.3em}    |p{5em}|p{5em}          |p{4.3em}|p{4.3em}    |p{4.3em}|p{4.3em}    |}
    \hline
    \multirow{2}{*}{$\Omega$}     &\multirow{2}{*}{$M$}          &\multirow{2}{*}{$\hat r$}     &\multicolumn{2}{l|}{Prony}                                   &\multicolumn{2}{l|}{Matrix Pencil}                           &\multicolumn{2}{l|}{Matrix Pencil SVD}                       &\multicolumn{2}{l|}{ESPRIT}                                  \\ \cline{4-11}
                                  &                              &                              &RMSE                          &INE                           &RMSE                          &INE                           &RMSE                          &INE                           &RMSE                          &INE                           \\ \hline
    $\{1,2,3,4\}$                 &60                            &20                            &$ 2.0 \cdot 10^{-7}$          &$ 7.7 \cdot 10^{-7}$          &$ 6.3 \cdot 10^{-7}$          &$ 2.7 \cdot 10^{-6}$          &$ 1.1 \cdot 10^{-7}$          &$ 3.9 \cdot 10^{-7}$          &$ 6.7 \cdot 10^{-2}$          &$ 1.3 \cdot 10^{-1}$          \\ \cline{4-11}
    $\{1,2,3,4\}$                 &80                            &20                            &$ 2.7 \cdot 10^{-7}$          &$ 1.1 \cdot 10^{-6}$          &$ 1.3 \cdot 10^{-7}$          &$ 5.8 \cdot 10^{-7}$          &$ 4.9 \cdot 10^{-7}$          &$ 2.2 \cdot 10^{-6}$          &$ 6.3 \cdot 10^{-2}$          &$ 1.3 \cdot 10^{-1}$          \\ \cline{4-11}
    $\{1,2,3,4\}$                 &100                           &20                            &$ 1.3 \cdot 10^{-6}$          &$ 5.7 \cdot 10^{-6}$          &$ 2.5 \cdot 10^{-7}$          &$ 1.0 \cdot 10^{-6}$          &$ 6.1 \cdot 10^{-7}$          &$ 1.7 \cdot 10^{-6}$          &$ 5.4 \cdot 10^{-2}$          &$ 1.2 \cdot 10^{-1}$          \\ \cline{4-11}
    $\{1,2,3,4\}$                 &120                           &20                            &$ 8.5 \cdot 10^{-7}$          &$ 3.7 \cdot 10^{-6}$          &$ 2.9 \cdot 10^{-6}$          &$ 1.3 \cdot 10^{-5}$          &$ 1.7 \cdot 10^{-6}$          &$ 7.6 \cdot 10^{-6}$          &$ 4.6 \cdot 10^{-2}$          &$ 8.8 \cdot 10^{-2}$          \\ \hline
    \end{tabular}
    \end{center}}
    \label{table:randwalk}
    \caption{\textbf{Example 2.} Errors of various algorithms and choices of $M$.}
\end{table}

    \paragraph{Example 3. Ring Graph\cite{perraudin2016gspbox}.}
    An undirected graph with 30 nodes uniformly distributed on a ring-shaped structure and each vertex has 8 neighbors. The edge weights are all equal to 1. The matrix $A$ in this example is a diagonalizable matrix with 13 distinct eigenvalues.  Unlike Example 1,  in this case, there is no repeating roots in the annihilating polynomial and therefore we can not recover their geometric multiplicities by our algorithms. We summarize the results below. 
    
    \begin{figure}[H]
    \includegraphics[width=0.4\linewidth]{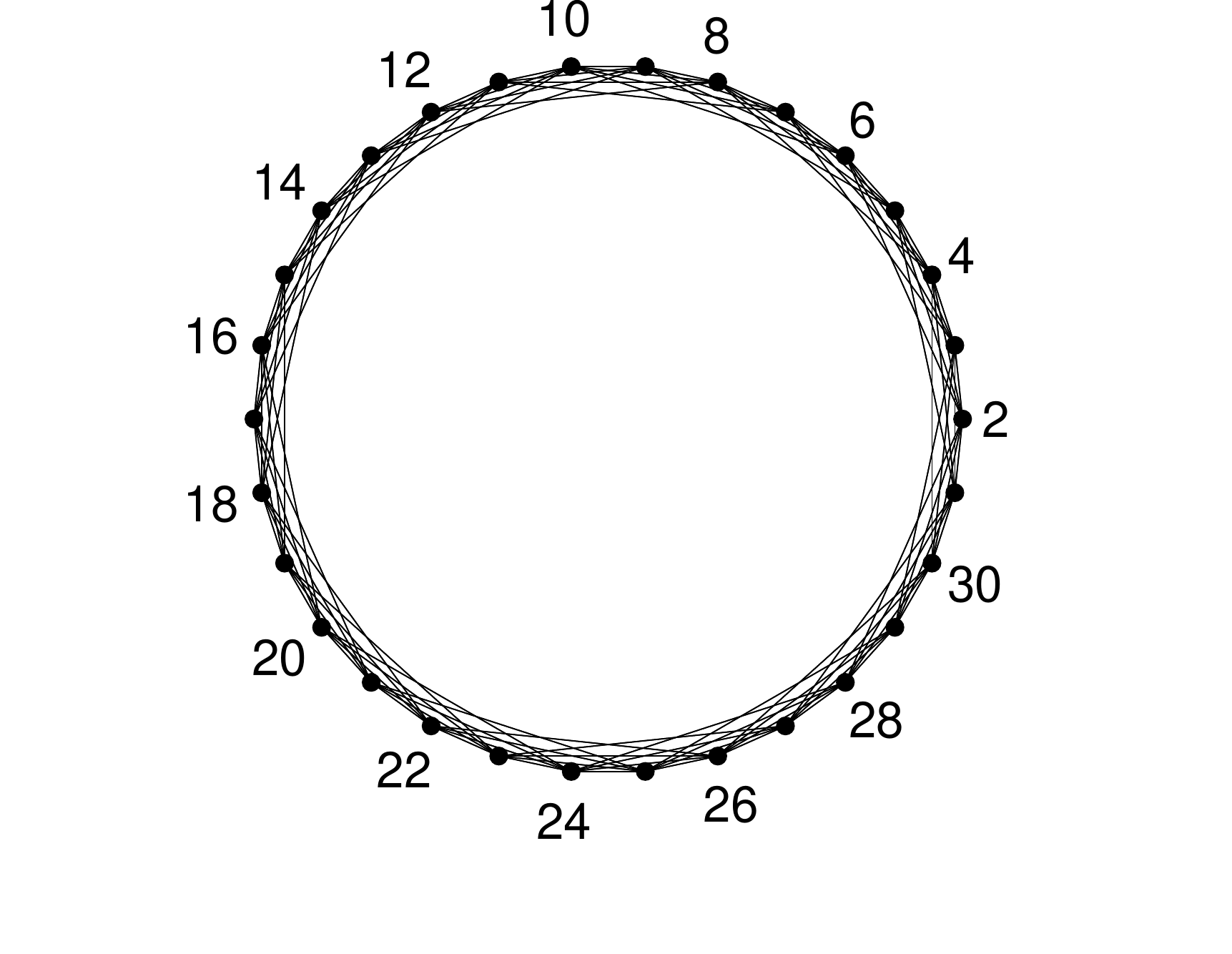}
    \includegraphics[width=0.6\linewidth]{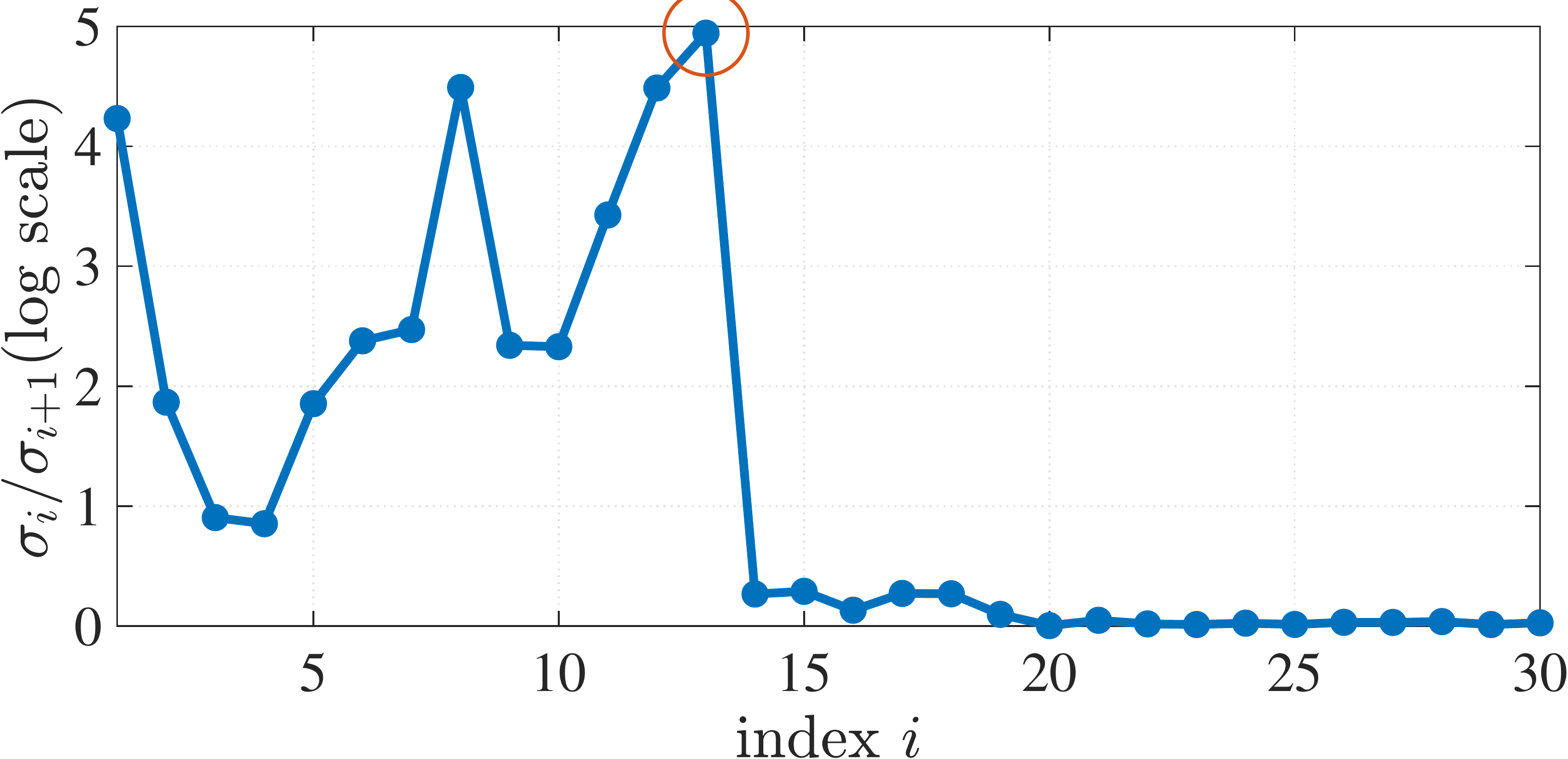}
    \caption{\textbf{Example 3.} (Left) An illustration of vertices and edges in the ring-shaped graph. (Right) Spectrum quotient of Hankel-type matrix for the choice of $\Omega=\{1,2,\dots,5\}$. In this case, 13 eigenvalues are $\Omega$-recoverable. }
    \end{figure}

\begin{table}[!h] 
    {\tiny
    \begin{center}
    \begin{tabular}
    {|p{5em}|p{1.2em}       |p{0.5em}     |p{4.3em}|p{4.3em}    |p{5em}|p{5em}          |p{4.3em}|p{4.3em}    |p{4.3em}|p{4.3em}    |}
        \hline
        \multirow{2}{*}{$\Omega$}     &\multirow{2}{*}{$M$}          &\multirow{2}{*}{$\hat r$}     &\multicolumn{2}{l|}{Prony}                                   &\multicolumn{2}{l|}{Matrix Pencil}                           &\multicolumn{2}{l|}{Matrix Pencil SVD}                       &\multicolumn{2}{l|}{ESPRIT}                                  \\ \cline{4-11}
                                      &                              &                              &RMSE                          &INE                           &RMSE                          &INE                           &RMSE                          &INE                           &RMSE                          &INE                           \\ \hline
        $\{1\}$                       &90                            &6                             &$ 6.5 \cdot 10^{-3}$          &$ 2.7 \cdot 10^{-2}$          &$ 1.8 \cdot 10^{-3}$          &$ 7.2 \cdot 10^{-3}$          &$ 6.0 \cdot 10^{-3}$          &$ 2.4 \cdot 10^{-2}$          &$ 5.7 \cdot 10^{-3}$          &$ 2.2 \cdot 10^{-2}$          \\ \cline{4-11}
        $\{1,2\}$                     &90                            &7                             &$ 4.0 \cdot 10^{-3}$          &$ 1.9 \cdot 10^{-2}$          &$ 2.2 \cdot 10^{-3}$          &$ 1.2 \cdot 10^{-2}$          &$ 4.2 \cdot 10^{-3}$          &$ 2.0 \cdot 10^{-2}$          &$ 1.5 \cdot 10^{-2}$          &$ 5.4 \cdot 10^{-2}$          \\ \cline{4-11}
        $\{1,2,3\}$                   &90                            &13                            &$ 3.7 \cdot 10^{-6}$          &$ 1.4 \cdot 10^{-5}$          &$ 1.5 \cdot 10^{-5}$          &$ 7.1 \cdot 10^{-5}$          &$ 1.2 \cdot 10^{-5}$          &$ 4.7 \cdot 10^{-5}$          &$ 5.2 \cdot 10^{-2}$          &$ 1.3 \cdot 10^{-1}$          \\ \cline{4-11}
        $\{1,2,3,4\}$                 &90                            &12                            &$ 3.6 \cdot 10^{-3}$          &$ 1.2 \cdot 10^{-2}$          &$ 1.3 \cdot 10^{-5}$          &$ 6.2 \cdot 10^{-5}$          &$ 4.1 \cdot 10^{-3}$          &$ 1.7 \cdot 10^{-2}$          &$ 1.6 \cdot 10^{-2}$          &$ 5.4 \cdot 10^{-2}$          \\ \cline{4-11}
        $\{1,2,3,4,5\}$               &90                            &13                            &$ 1.8 \cdot 10^{-7}$          &$ 8.4 \cdot 10^{-7}$          &$ 2.3 \cdot 10^{-6}$          &$ 1.2 \cdot 10^{-5}$          &$ 1.7 \cdot 10^{-6}$          &$ 8.8 \cdot 10^{-6}$          &$ 1.8 \cdot 10^{-2}$          &$ 6.3 \cdot 10^{-2}$          \\ \hline
    \end{tabular}
    \end{center}}
    \label{table:randwalk}
    \caption{\textbf{Example 3.} Numerical rank and errors of various algorithms and choices of $\Omega$.}
\end{table}

\begin{table}[!h] 
        
    {\tiny
    \begin{center}
        \begin{tabular}
        {|p{5em}|p{1.2em}       |p{0.5em}     |p{4.3em}|p{4.3em}    |p{5em}|p{5em}          |p{4.3em}|p{4.3em}    |p{4.3em}|p{4.3em}    |}
            \hline
            \multirow{2}{*}{$\Omega$}     &\multirow{2}{*}{$M$}          &\multirow{2}{*}{$\hat r$}     &\multicolumn{2}{l|}{Prony}                                   &\multicolumn{2}{l|}{Matrix Pencil}                           &\multicolumn{2}{l|}{Matrix Pencil SVD}                       &\multicolumn{2}{l|}{ESPRIT}                               \\ \cline{4-11}
                                          &                              &                              &RMSE                          &INE                           &RMSE                          &INE                           &RMSE                          &INE                           &RMSE                          &INE                         \\ \hline
            $\{1,2,3,4,5\}$               &90                            &13                            &$ 1.8 \cdot 10^{-7}$          &$ 8.4 \cdot 10^{-7}$          &$ 2.3 \cdot 10^{-6}$          &$ 1.2 \cdot 10^{-5}$          &$ 1.7 \cdot 10^{-6}$          &$ 8.8 \cdot 10^{-6}$          &$ 1.8 \cdot 10^{-2}$          &$ 6.3 \cdot 10^{-2}$          \\ \cline{4-11}
            $\{1,2,3,4,5\}$               &120                           &13                            &$ 1.2 \cdot 10^{-6}$          &$ 5.9 \cdot 10^{-6}$          &$ 3.0 \cdot 10^{-6}$          &$ 1.4 \cdot 10^{-5}$          &$ 2.4 \cdot 10^{-6}$          &$ 1.2 \cdot 10^{-5}$          &$ 1.4 \cdot 10^{-2}$          &$ 4.8 \cdot 10^{-2}$          \\ \cline{4-11}
            $\{1,2,3,4,5\}$               &150                           &13                            &$ 2.2 \cdot 10^{-6}$          &$ 1.1 \cdot 10^{-5}$          &$ 1.0 \cdot 10^{-6}$          &$ 5.4 \cdot 10^{-6}$          &$ 3.1 \cdot 10^{-6}$          &$ 1.5 \cdot 10^{-5}$          &$ 1.1 \cdot 10^{-2}$          &$ 3.0 \cdot 10^{-2}$          \\ \cline{4-11}
            $\{1,2,3,4,5\}$               &180                           &13                            &$ 5.1 \cdot 10^{-7}$          &$ 2.1 \cdot 10^{-6}$          &$ 2.8 \cdot 10^{-6}$          &$ 1.4 \cdot 10^{-5}$          &$ 3.1 \cdot 10^{-6}$          &$ 1.6 \cdot 10^{-5}$          &$ 8.3 \cdot 10^{-3}$          &$ 2.4 \cdot 10^{-2}$          \\ \hline
        \end{tabular}
    
    \end{center}
    }
    \label{table:randwalk}
    \caption{\textbf{Example 3.} Errors of various algorithms and choices of $M$. }
\end{table}

\paragraph{(b) Heat diffusion process over graphs.}

We consider a non-homogeneous heat diffusion process over the graph $\mathcal{G}$ is  governed by  a continuous affine system
\begin{align}
x_{t+1}^{cont}&=e^{-t\mathcal{L}}x_t^{cont}+c, t\geq 0\label{linearsystem}.
\end{align} We observe the system at uniform time instances $0,\Delta t, 2\Delta t,\cdots$. 
In this subsection, we explore  the reconstruction of eigenvalues when the system matrix is relatively large. 

% \begin{wrapfigure}{r}{0.4\linewidth}
%     \vspace{-0.2in}
%     \includegraphics[width=\linewidth]{./figures/sphere_graph-crop.pdf}
%     \vspace{-0.8in}
% \end{wrapfigure}
\paragraph{Example 4. Sphere Graph\cite{perraudin2016gspbox}.} An undirected graph with 150 nodes sampled on a hyper-sphere and each vertex is connected to its 10 nearest neighbors. In this example,   $x_t^{cont}(0)$ is randomly generated from the uniform distribution and $c$ is in the image of a random Gaussian vector under the map $e^{-t\mathcal{L}}$. We observe the system at $t_l=l\Delta t$ for $l=0,\cdots,M-1$ and $\Delta t=20$. The matrix $A = e^{-20\mathcal{L}}$ is approximately low rank: only 10 eigenvalues are greater than $10^{-3}$(see the left bottom panel of Figure \ref{spheregraph}).  The largest gap happens between 7th eigenvalue  0.0455 and 8th eigenvalue 0.0097. We summarize the spectral plot of  the numerical rank estimation (the top right panel of Figure \ref{spheregraph}) and  reconstruction results in the Table 7 (various $\Omega$) and Table 8 (various $M$). It shows that our algorithms can recover significant eigenvalues very well.  We also 
 investigate various choice of $\Delta t$. When $\Delta t$ is relatively small,  all eigenvalues of $A$  lie in $[0,1]$ and form clusters. Two eigenvalues in the same cluster are very close to each other and can be identified with the same value ''numerically". Below, we show that our algorithms can recover the representative eigenvalues in each cluster (the right bottom panel of Figure \ref{spheregraph}). 

\begin{figure}[H]
\includegraphics[width=0.4\linewidth]{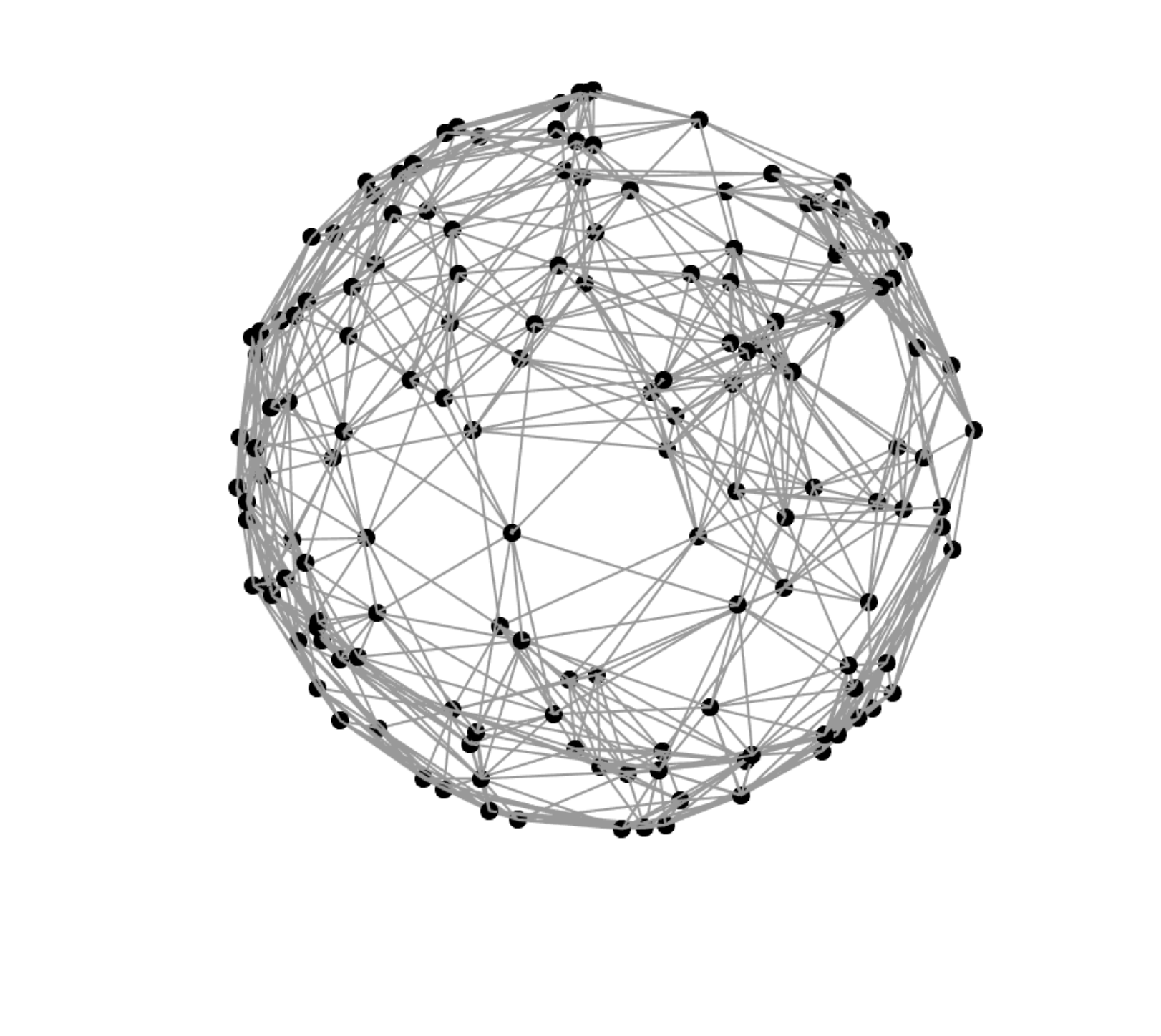}
\includegraphics[width=0.6\linewidth]{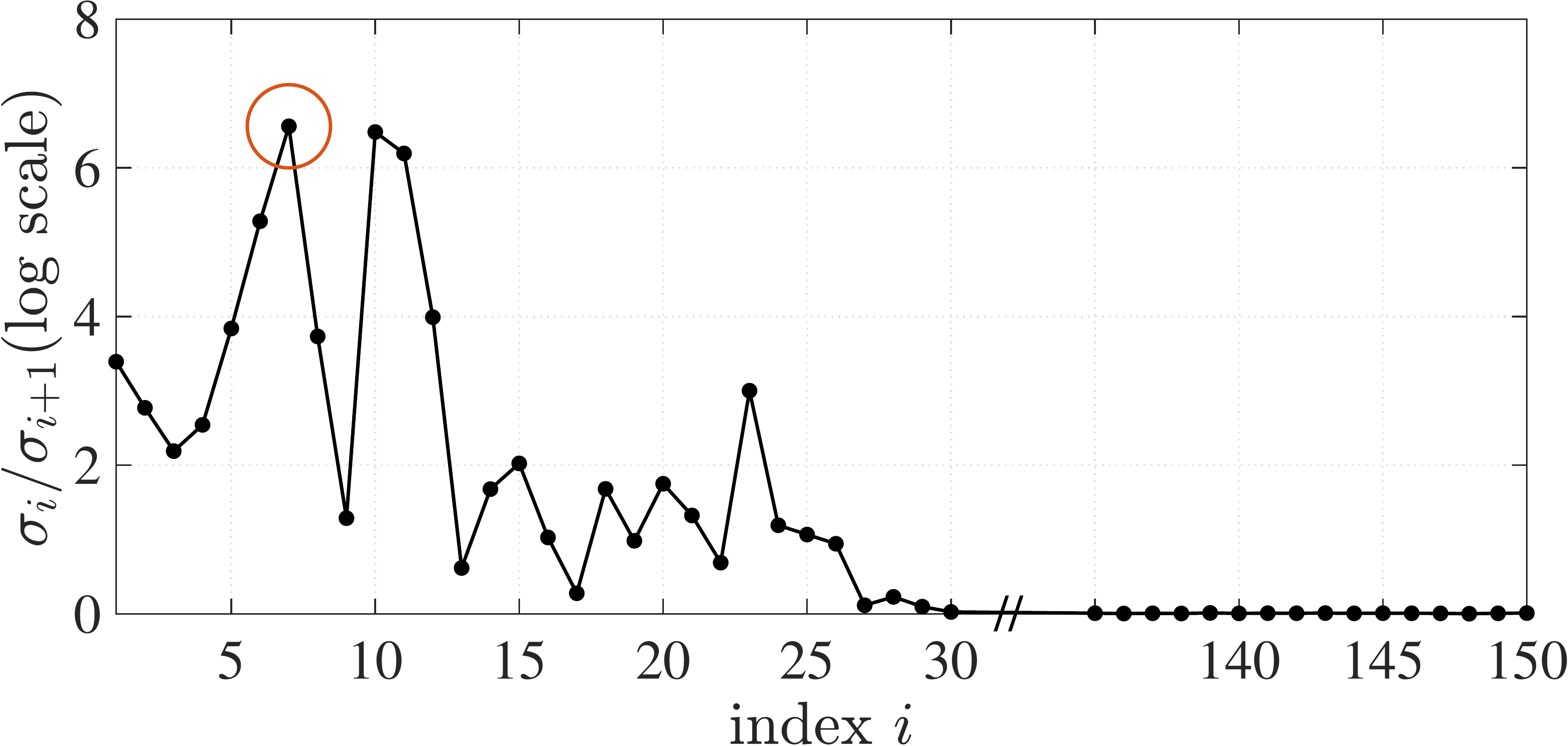}
\end{figure}
\vspace{-0.8cm}
\begin{figure}[H]
\centering
\includegraphics[width=0.35\linewidth]{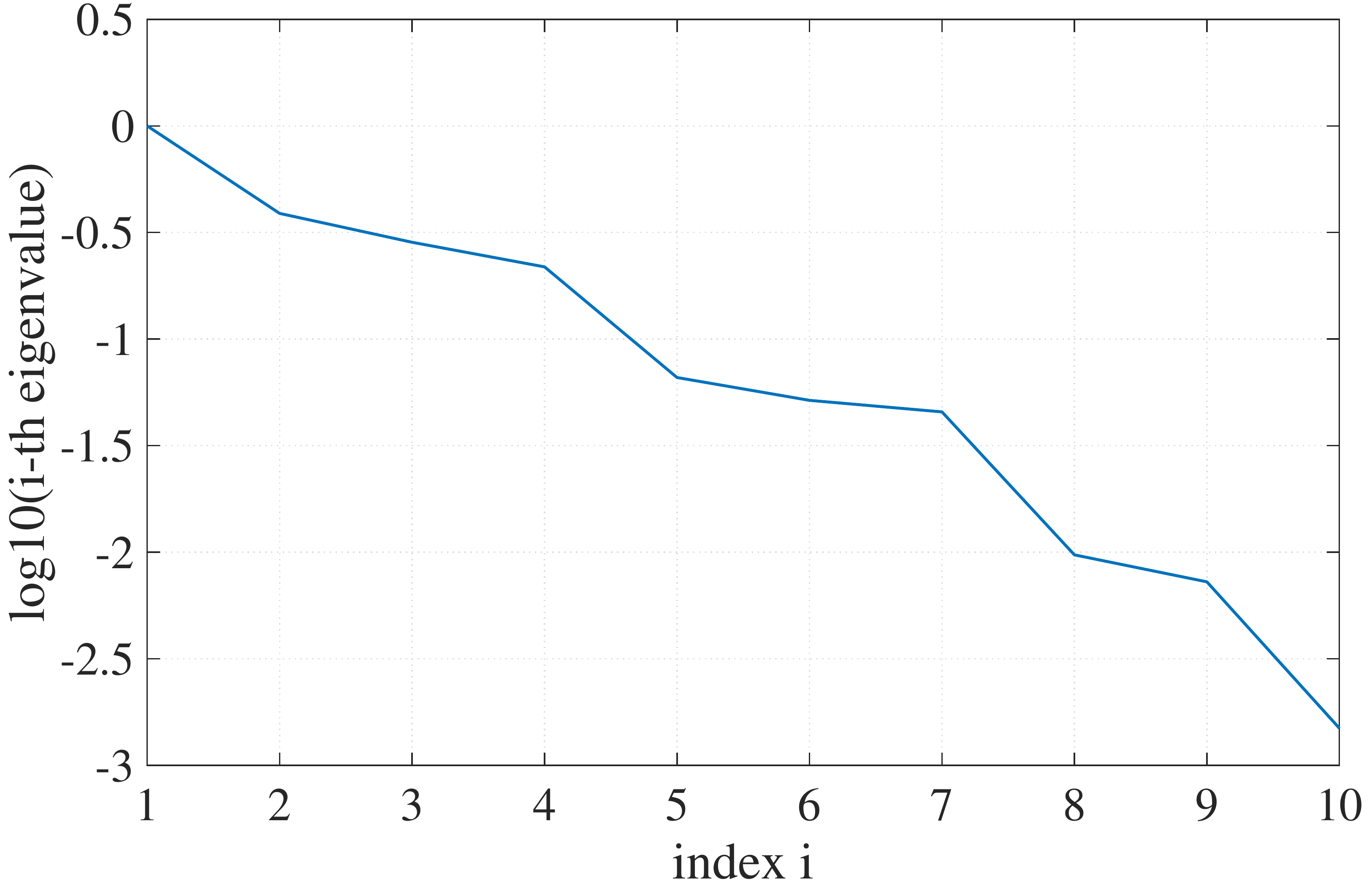}
%\hspace{1mm}
\includegraphics[width=0.55\linewidth]{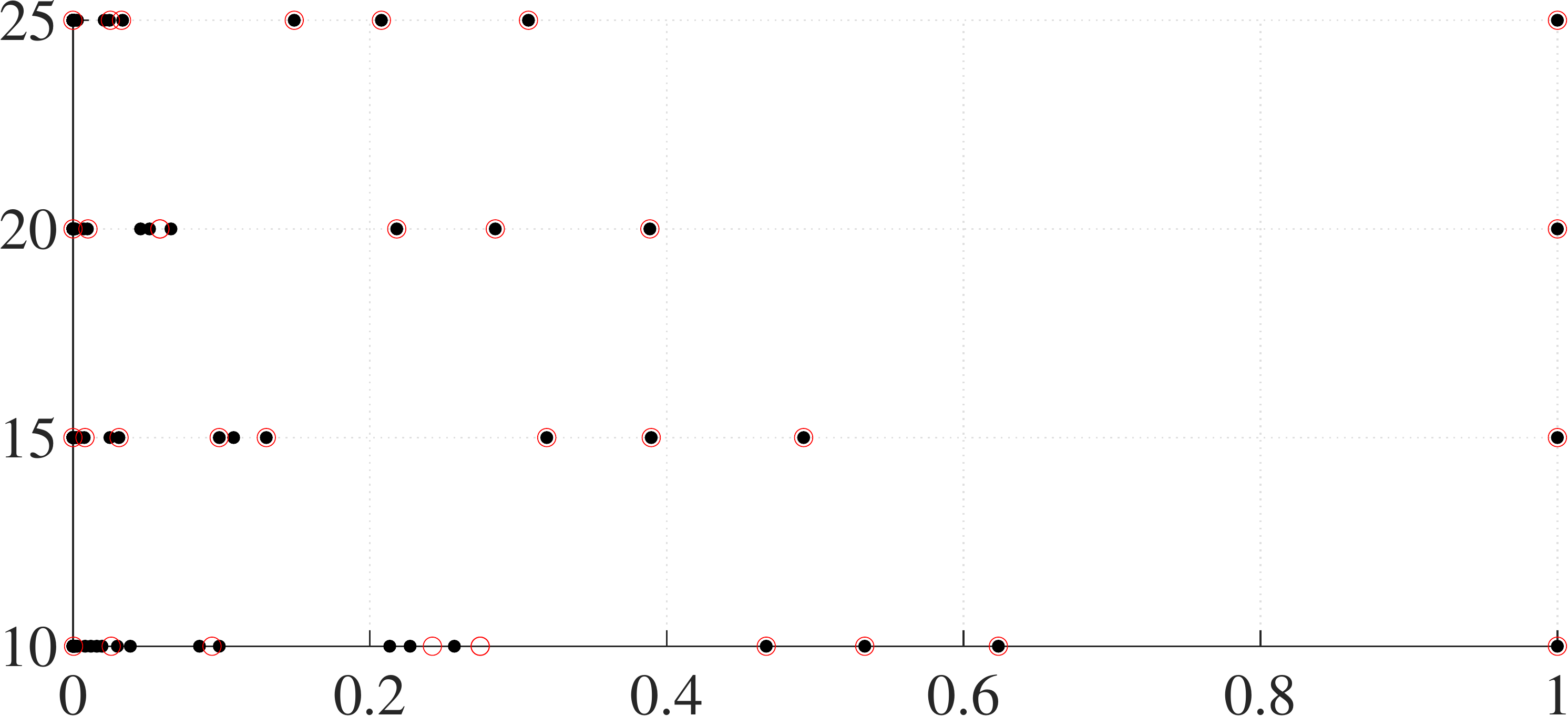}
\caption{\textbf{Example 4.} (Top Left) An illustration of vertices and edges in the sphere-shaped graph. (Top Right) Spectrum quotient of Hankel-type matrix for the choice of $\Omega=\{1,2,3,4\}$. In this case, 7 eigenvalues are $\Omega$-recoverable. (Bottom left)  The eigenvalue of $A$ when $\Delta t =20$.  (Bottom right)The ground truth spectrum (filled points) and reconstructed spectrum (hollow points) for  $\Delta t=10:5:25$ where we set the rank $\hat r=10$ in all cases.}\label{spheregraph}
\end{figure}

\begin{table}[!h] 
    {\tiny
    \begin{center}
    \begin{tabular}
    {|p{5em}|p{1.2em}       |p{0.5em}     |p{4.3em}|p{4.3em}    |p{5em}|p{5em}          |p{4.3em}|p{4.3em}    |p{4.3em}|p{4.3em}    |}
    \hline
\multirow{2}{*}{$\Omega$}     &\multirow{2}{*}{$M$}          &\multirow{2}{*}{$\hat r$}     &\multicolumn{2}{l|}{Prony}                                   &\multicolumn{2}{l|}{Matrix Pencil}                           &\multicolumn{2}{l|}{Matrix Pencil SVD}                       &\multicolumn{2}{l|}{ESPRIT}                                  \\ \cline{4-11}
                              &                              &                              &RMSE                          &INE                           &RMSE                          &INE                           &RMSE                          &INE                           &RMSE                          &INE                           \\ \hline
$\{1,2\}$                     &450                           &5                             &$ 8.8 \cdot 10^{-4}$          &$ 8.3 \cdot 10^{-3}$          &$ 2.4 \cdot 10^{-4}$          &$ 2.8 \cdot 10^{-3}$          &$ 8.8 \cdot 10^{-4}$          &$ 8.3 \cdot 10^{-3}$          &$ 7.3 \cdot 10^{-4}$          &$ 6.3 \cdot 10^{-3}$          \\ \cline{4-11}
$\{1,2,3\}$                   &450                           &6                             &$ 1.5 \cdot 10^{-4}$          &$ 1.8 \cdot 10^{-3}$          &$ 1.8 \cdot 10^{-3}$          &$ 1.6 \cdot 10^{-2}$          &$ 1.5 \cdot 10^{-4}$          &$ 1.8 \cdot 10^{-3}$          &$ 1.5 \cdot 10^{-4}$          &$ 1.8 \cdot 10^{-3}$          \\ \cline{4-11}
$\{1,2,3,4\}$                 &450                           &7                             &$ 7.8 \cdot 10^{-4}$          &$ 9.5 \cdot 10^{-3}$          &$ 6.8 \cdot 10^{-4}$          &$ 7.7 \cdot 10^{-3}$          &$ 7.8 \cdot 10^{-4}$          &$ 9.4 \cdot 10^{-3}$          &$ 7.6 \cdot 10^{-4}$          &$ 9.2 \cdot 10^{-3}$          \\ \cline{4-11}
$\{1,2,3,4,5\}$               &450                           &7                             &$ 8.1 \cdot 10^{-4}$          &$ 9.8 \cdot 10^{-3}$          &$ 9.4 \cdot 10^{-4}$          &$ 1.1 \cdot 10^{-2}$          &$ 8.1 \cdot 10^{-4}$          &$ 9.8 \cdot 10^{-3}$          &$ 7.8 \cdot 10^{-4}$          &$ 9.5 \cdot 10^{-3}$          \\ \cline{4-11}
\hline
    \end{tabular}
    \end{center}}
    \label{table:spherewalk}
    \caption{\textbf{Example 4.} Numerical rank and errors of various algorithms and choices of $\Omega$ for $\Delta t=20$.}
\end{table}

\begin{table}[!h] 
    {\tiny
    \begin{center}
    \begin{tabular}
    {|p{5em}|p{1.2em}       |p{0.5em}     |p{4.3em}|p{4.3em}    |p{5em}|p{5em}          |p{4.3em}|p{4.3em}    |p{4.3em}|p{4.3em}    |}
    \hline
 \multirow{2}{*}{$\Omega$}     &\multirow{2}{*}{$M$}          &\multirow{2}{*}{$\hat r$}     &\multicolumn{2}{l|}{Prony}                                   &\multicolumn{2}{l|}{Matrix Pencil}                           &\multicolumn{2}{l|}{Matrix Pencil SVD}                       &\multicolumn{2}{l|}{ESPRIT}                                  \\ \cline{4-11}
                              &                              &                              &RMSE                          &INE                           &RMSE                          &INE                           &RMSE                          &INE                           &RMSE                          &INE                           \\ \hline
$\{1,2,3,4,5\}$               &450                           &7                             &$ 8.1 \cdot 10^{-4}$          &$ 9.8 \cdot 10^{-3}$          &$ 9.4 \cdot 10^{-4}$          &$ 1.1 \cdot 10^{-2}$          &$ 8.1 \cdot 10^{-4}$          &$ 9.8 \cdot 10^{-3}$          &$ 7.8 \cdot 10^{-4}$          &$ 9.5 \cdot 10^{-3}$          \\ \cline{4-11}
$\{1,2,3,4,5\}$               &600                           &7                             &$ 8.1 \cdot 10^{-4}$          &$ 9.8 \cdot 10^{-3}$          &$ 9.4 \cdot 10^{-4}$          &$ 1.1 \cdot 10^{-2}$          &$ 8.1 \cdot 10^{-4}$          &$ 9.8 \cdot 10^{-3}$          &$ 7.8 \cdot 10^{-4}$          &$ 9.5 \cdot 10^{-3}$          \\ \cline{4-11}
$\{1,2,3,4,5\}$               &750                           &7                             &$ 8.1 \cdot 10^{-4}$          &$ 9.8 \cdot 10^{-3}$          &$ 9.4 \cdot 10^{-4}$          &$ 1.1 \cdot 10^{-2}$          &$ 8.1 \cdot 10^{-4}$          &$ 9.8 \cdot 10^{-3}$          &$ 7.8 \cdot 10^{-4}$          &$ 9.5 \cdot 10^{-3}$          \\ \cline{4-11}
$\{1,2,3,4,5\}$               &900                           &7                             &$ 8.1 \cdot 10^{-4}$          &$ 9.8 \cdot 10^{-3}$          &$ 9.4 \cdot 10^{-4}$          &$ 1.1 \cdot 10^{-2}$          &$ 8.1 \cdot 10^{-4}$          &$ 9.8 \cdot 10^{-3}$          &$ 7.8 \cdot 10^{-4}$          &$ 9.5 \cdot 10^{-3}$          \\ \cline{4-11}
\hline
    \end{tabular}
    \end{center}}
    \label{table:randwalk}
    \caption{\textbf{Example 4.} Errors of various algorithms and choices of $M$.}
\end{table}

In the last subsection, we consider the case where the time series data can be well-approximated by a discrete linear dynamical system governed by $A$. We use partial observations of the original time series data and compare the reconstruction results with the eigenvalues of $A$. 
\paragraph{Example 5. Non-linear LIP model.}
% \subsubsection{Application to non-linear LIP model}
In this example, we consider a 3 dimensional discrete homogeneous dynamical system that serves as an approximation to the LIP model of influenza virus inflection model: 
\begin{equation}
    \begin{cases}
        \dot V = r I - c V \\
        \dot H = - \beta H V \\
        \dot I = \beta H V - \delta I
    \end{cases}
\end{equation}
As mentioned in paper \cite{duan2020identification}, we used rescaled model for convenience, with parameters $\beta=10.8$, $r=12$, $c=3$, $\delta=4$ and initial state $V(0) = 0.093/(4\times 10^5)$, $H(0)=1$, $I(0)=0$. We use Matlab built-in function ode45 to obtain the solution.

We consider the time series data when $t\in[1.8,2.1]$ as the underlying noise-free observations (with $\Delta t=0.01$), where the number of uninfected cells is rapidly decreasing due to the increasing number of infected cells and their released virus (see the top right panel of Figure \ref{Figure5}). We consider the least square approximation to obtain a discrete homogeneous linear system  as  
\begin{equation}
    A = \begin{bmatrix}\vline&\vline&&\vline\\ {x}_0&{x}_1&\cdots&{x}_{N-2}\\ \vline&\vline&&\vline\end{bmatrix}^\dagger\begin{bmatrix}\vline&\vline&&\vline\\ {x}_1&{x}_2&\cdots&{x}_{N-1}\\ \vline&\vline&&\vline\end{bmatrix}
\end{equation} 
where $x_t$ for $t=0,1,\dots,N-1$ are the discrete time-state observation and $N=30$. Then we simulate the linear dynamical system with the initial state as $x_0$ and the system matrix $A$. The relative mean squared error is of $O(1e-3)$. Algorithms 1-4 are performed with the original data from LIP, and we compare the estimated eigenvalues with those of $A$.  The result is presented as below.
\begin{figure}[!h]
    \includegraphics[width=0.65\linewidth]{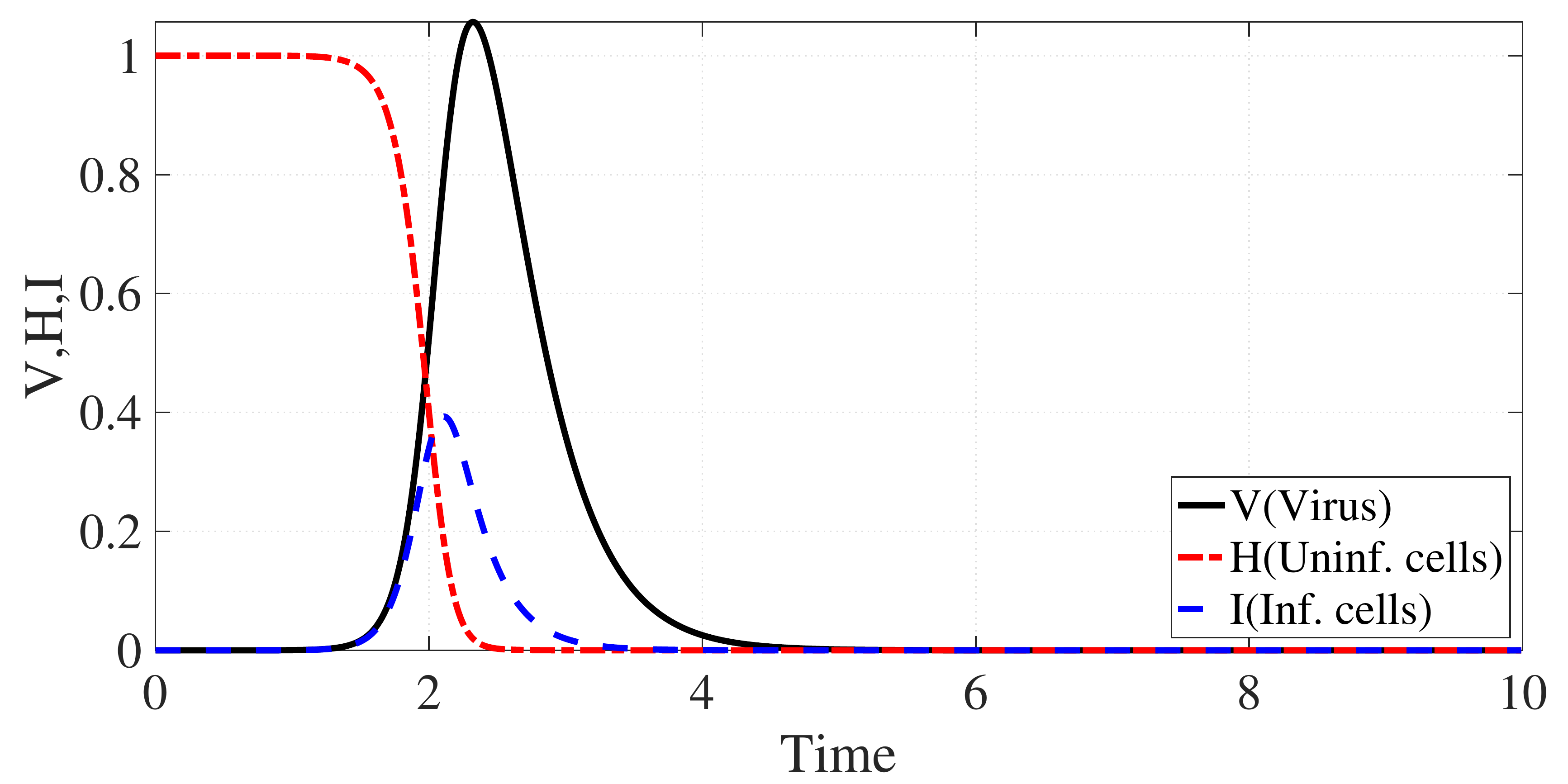}
    \includegraphics[width=0.3\linewidth]{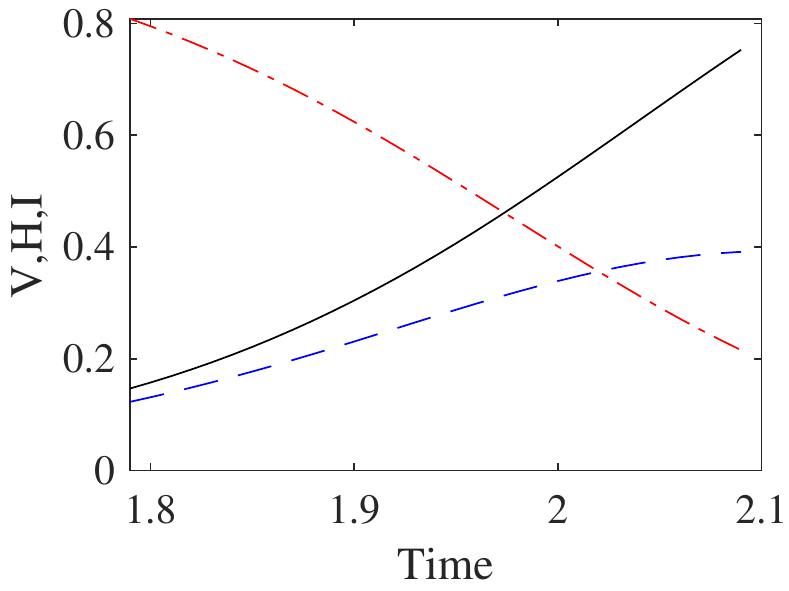}
    \centering
    \caption{\textbf{Example 5.} Virus(V) infects susceptible cells(H) with rate $\beta$. Infected cells are cleared with rate $\delta$. Once cells are productively infected (I), they release virus at rate $r$ and virus are cleared at rate $c$. The susceptible cells (red line) are rapidly infected while the virus (black line) and infected cells (blue line) peak at $t=2.2$ approximately. The viral growth is limited by the number of susceptible cells, decreasing the viral load and the number of infected cells to undetectable levels.}\label{Figure5}
\end{figure}

\begin{table}[H] 
    {\tiny
    \begin{center}
    \begin{tabular}
        {|p{4em}|p{2.2em}       |p{0.5em}     |p{4.3em}|p{4.3em}    |p{5em}|p{5em}          |p{4.3em}|p{4.3em}    |p{4.3em}|p{4.3em}    |}
        \hline
        \multirow{2}{*}{$\Omega$}     &\multirow{2}{*}{$M$}          &\multirow{2}{*}{$\hat r$}     &\multicolumn{2}{l|}{Prony}                                   &\multicolumn{2}{l|}{Matrix Pencil}                           &\multicolumn{2}{l|}{Matrix Pencil SVD}                       &\multicolumn{2}{l|}{ESPRIT}                                  \\ \cline{4-11}
                                      &                              &                              &RMSE                          &INE                           &RMSE                          &INE                           &RMSE                          &INE                           &RMSE                          &INE                           \\ \hline
        $\{1\}$                       &30                            &3                             &$ 7.7 \cdot 10^{-13}$         &$ 1.3 \cdot 10^{-12}$         &$ 2.7 \cdot 10^{-13}$         &$ 4.2 \cdot 10^{-13}$         &$ 9.7 \cdot 10^{-13}$         &$ 1.6 \cdot 10^{-12}$         &$ 8.8 \cdot 10^{-14}$         &$ 1.5 \cdot 10^{-13}$         \\ \cline{4-11}
        $\{1,2\}$                     &30                            &3                             &$ 1.4 \cdot 10^{-13}$         &$ 1.4 \cdot 10^{-13}$         &$ 1.1 \cdot 10^{-13}$         &$ 1.1 \cdot 10^{-13}$         &$ 7.7 \cdot 10^{-14}$         &$ 8.1 \cdot 10^{-14}$         &$ 5.5 \cdot 10^{-2}$          &$ 6.6 \cdot 10^{-2}$          \\ \cline{4-11}
        $\{1,2,3\}$                   &30                            &3                             &$ 1.8 \cdot 10^{-13}$         &$ 2.0 \cdot 10^{-13}$         &$ 4.4 \cdot 10^{-14}$         &$ 5.3 \cdot 10^{-14}$         &$ 1.7 \cdot 10^{-13}$         &$ 2.1 \cdot 10^{-13}$         &$ 6.4 \cdot 10^{-2}$          &$ 7.7 \cdot 10^{-2}$          \\ \hline
        
    \end{tabular}
        
    \end{center}}
    \caption{ \textbf{Example 5.} Numerical rank and  errors of various algorithms and choices of $\Omega$.}
    \label{tab:LIP}
\end{table}

\paragraph{Example 6. Human Walk Motion.}

In this example, we consider the captured motion of a walking human. \footnote{The data set was
obtained from mocap.cs.cmu.edu, which is the first trial of CMU Mocap subject 07 and created with funding
from NSF EIA-0196217.}  We used the one that is available in \href{https://github.com/luckystarufo/DMD_for_Human_motion}{Yuying Liu's github}. The  data set is of size $54 \times 316$, which is collected by 18 sensors and each sensor provides $(x,y,z)$ information of motion. The total number of time frames is 316 (the gap between each time frame is 1/120 seconds). We first normalize the real data set so their mean is 0 and standard deviation is 1.  We then apply the Dynamic Decomposition method  to obtain $A$. The relative mean square error is $O(10^{-3})$. We use the real motion data in our algorithm and compare the outputs with the eigenvalues of $A$.  We report its result as follows. We see that Matrix-Pencil-SVD algorithm has the best performance.

\begin{table}[H] 
    {\tiny
    \begin{center}
    \begin{tabular}
        {|p{5em}|p{1.2em}       |p{0.5em}     |p{4.3em}|p{4.3em}    |p{5em}|p{5em}          |p{4.3em}|p{4.3em}    |p{4.3em}|p{4.3em}    |}
        \hline
  \multirow{2}{*}{$\Omega$}     &\multirow{2}{*}{$M$}          &\multirow{2}{*}{$\hat r$}     &\multicolumn{2}{l|}{Prony}                                   &\multicolumn{2}{l|}{Matrix Pencil}                           &\multicolumn{2}{l|}{Matrix Pencil SVD}                       &\multicolumn{2}{l|}{ESPRIT}                                  \\ \cline{4-11}
                              &                              &                              &RMSE                          &INE                           &RMSE                          &INE                           &RMSE                          &INE                           &RMSE                          &INE                           \\ \hline
$\{1\}$                       &316                           &10                            &$ 2.2 \cdot 10^{-1}$          &$ 7.9 \cdot 10^{-1}$          &$ 2.3 \cdot 10^{-1}$          &$ 9.8 \cdot 10^{-1}$          &$ 1.6 \cdot 10^{-2}$          &$ 5.7 \cdot 10^{-2}$          &$ 2.0 \cdot 10^{-2}$          &$ 9.1 \cdot 10^{-2}$          \\ \cline{4-11}
$\{1,3,9\}$                   &316                           &12                            &$ 2.6 \cdot 10^{-1}$          &$ 8.3 \cdot 10^{-1}$          &$ 3.5 \cdot 10^{-1}$          &$ 9.7 \cdot 10^{-1}$          &$ 1.9 \cdot 10^{-2}$          &$ 7.2 \cdot 10^{-2}$          &$ 3.6 \cdot 10^{-2}$          &$ 1.6 \cdot 10^{-1}$          \\ \cline{4-11}
$\{1,2,3,9\}$                 &316                           &11                            &$ 2.4 \cdot 10^{-1}$          &$ 7.7 \cdot 10^{-1}$          &$ 3.5 \cdot 10^{-1}$          &$ 9.7 \cdot 10^{-1}$          &$ 2.6 \cdot 10^{-2}$          &$ 1.3 \cdot 10^{-1}$          &$ 2.2 \cdot 10^{-2}$          &$ 8.8 \cdot 10^{-2}$          \\ \cline{4-11}
$\{1,2,3,5,9\}$               &316                           &12                            &$ 2.4 \cdot 10^{-1}$          &$ 7.3 \cdot 10^{-1}$          &$ 3.5 \cdot 10^{-1}$          &$ 9.7 \cdot 10^{-1}$          &$ 2.6 \cdot 10^{-2}$          &$ 9.0 \cdot 10^{-2}$          &$ 3.2 \cdot 10^{-2}$          &$ 1.4 \cdot 10^{-1}$          \\ \cline{4-11}
\hline
    \end{tabular}
    \end{center}
        }
    \caption{\textbf{Example 6.} Root Mean square errors and infinity errors for various algorithms.  In this table, $\hat r$ is the estimator of algorithm 1 using (i) in step 2 with relative threshold $\epsilon=10^{-3}$.}
    \label{tab:human-motion}
\end{table}

\vspace{-0.3in}
\paragraph{Summary.}
In the case of noise-free data, we have seen that  all algorithms provide accurate reconstruction of  the eigenvalues  of various affine systems in $\Lambda$, matching our theoretical results developed in Section \ref{sec2:Main-Results}. They can also provide  faithful approximation of eigenvalues when the time series data  is approximately generated by a linear system. For comparison of algorithms, we found that the the Matrix Pencil method has the best performance in most cases when $|\Omega|$ is small and in the real human motion data, demonstrating its robustness.  When $|\Omega|$ increases and the data is exact, the Prony method has the best performance in most cases. In particular, 
\begin{itemize}
\item  As $|\Omega|$ increases, more eigenvalues could be recovered and the reconstruction accuracy got improved until $|\Omega|$ is sufficiently large. This is because larger  $|\Omega|$ would yield a larger matrix that is sensitive to perturbation caused by numerical round off errors.  

\item  For the temporal steps,  setting $M=3d$ is typically enough since the dynamical systems considered in our numerical examples have reached equilibrium (See the state plot in Example 1). In these cases, increasing more temporal samples does not bring any new information and did not help in improving the accuracy. An exception case is that the performance of ESPRIT still got improved.

\item The difficulty that affect the reconstruction accuracy for  large affine systems is caused by the clustering phenomenon of eigenvalues, i.e., the minimal gap between eigenvalues is close to zero. The round-off errors  prevented us from recovering  eigenvalues that prove in theory, but the numerical example in sphere graph shows that we are able to recover significant eigenvalues or representative eigenvalue in a cluster. 

\item Finally, we have also numerically demonstrate that the algebraic multiplicities of eigenvalues can be recovered when $A$ is similar to a Jordan matrix (see Example 1), while we lost this information when $A$ is diagonalizable (see Example 3). 
\end{itemize}

\section{Conclusion and future work}

This paper studied estimating the spectrum of affine systems from partial observations of a single trajectory data. We derived various characterizations on the interplay among the observational locations, the behavior of the dynamical systems, and the spectral properties of the system matrix for the recoverability of eigenvalues. We propose several algorithms, which allow the usage of space-time samples with flexible temporal length and have been applied to a wide variety of examples on both synthetic and real data sets

Several interesting questions are left for future investigations. First, we would like to address the ``optimal selection" of observational locations. Given a fixed number of sensors, some choices of locations perform better than others in terms of numerical stability. We would like to find a characterization. Second, devising  denoising algorithm when the observation data is corrupted by noise. One direction we would like to pursue is, when we have multiple observation locations, how to make use of the temporal correlation between them to denoise the Hankel-type matrix. Third, we only consider the recovery of eigenvalues in this paper. It would be interesting to explore when the corresponding eigenspace projections can also be recovered. This problem is related to the completion of the low-rank matrix. We would like to derive conditions under which the recovery is feasible and propose robust algorithms to find a faithful approximation to the original system. 

\section{Appendix}

	% lemma 1
	\begin{lemma}\label{lemma1}
	Let $A=UJU^{-1}$ be its Jordan decomposition as in \eqref{Jordan1} and $\{V_s\}_{s=1}^{n}$ be its corresponding invariant subspaces. For any $b \in \mathbb{C}^d$,   we have $q_b^A = \prod\limits_{s=1}^n q_{b_s}^A$, where $b_s= P(\lambda_s; A)b\in \mathrm{V}_s$ with $P(\lambda_s;A)$ the projection onto $V_s$ and  the polynomials $\{q_{b_s}^A\}_{s=1}^{n}$ are coprime with each other.
	\end{lemma}

	\begin{proof}
		One one hand, we prove that $\prod \limits_{s=1}^n q_{b_s}^A$ divides $q_b^A$.
 Note the annihilating polynomials of $A$ of $b_s$ form an ideal generated by $q_{b_s}^A$:
		\begin{equation}\label{ideal1}
		\mathcal{I}^{A}_{b_s} := \{p(z) \in \mathbb{C}[z]\ |\ p(A)(b_s) = 0\} = \langle q_{b_s}^A\rangle.
		\end{equation}
		
	Denote by $A_{\big|V_s}$	 the restriction of $A$ on the invariant space $V_s$ and by $A_s$ its matrix form under standard basis.  Let $q^{A_s}(t) =(t-\lambda_s)^{r^{A_s}} $ denote the minimal polynomial of $A_s$ and then $q^{A_s}(A_s)\equiv 0$, which implies 
		
		\begin{equation}
		    q^{A_s}(A_s)b_s=0,
		\end{equation}
		so from the ideal (\ref{ideal1}) property it follows that $q_{b_s}^A$ divides $q^{A_s}$. So $q_{b_s}^{A}(z) = (z-\lambda_i)^{r_{b_s}^A}$. Since eigenvalues $\{\lambda_s\}_{s=1}^{n}$ are distinct, we know that $q_{b_s}^A$ are coprime to each other.
		
		Note that $\sum\limits_{s=1}^n q_{b}^A(A)b_s=q_b^A(A)b=0$ and $q_b^A(A)b_s \in V_s$. It follows that $q_{b}^A(A)b_s= 0$.  By the property of  ideal (\ref{ideal1}),  we  know that $q_{b_s}^A$ divides $q_b^A$ for $s=1,\cdots,n$. Therefore,  $\prod \limits_{s=1}^n q_{b_s}^A$ divides $q_b^A$.
		
		On the other hand, we show that $q_b^A$ divides $\prod \limits_{s=1}^n q_{b_s}^A$. Define 		
		\begin{equation}
		\mathcal{I}^{A}_{b} = \{p \in \mathbb{C}[z]\ |\ p(A)b = 0\} = \langle q_{b}^A\rangle.
		\end{equation}
		
		Note that 
		\begin{equation}
		    \prod \limits_{s=1}^n q_{b_s}^A(A)b= \prod \limits_{s=1}^n q_{b_s}^A(A)\sum\limits_{s=1}^{n} b_s =\sum\limits_{s=1}^{n} \prod \limits_{s=1}^n q_{b_s}^A(A) b_s = 0,
		\end{equation}
		it follows that $q_b^A$ divides $\prod \limits_{i=1}^n q_{b_i}^A$. Hence the conclusion follows. 
	\end{proof}

	% lemma 2
	\begin{lemma}\label{lemma2}
	For $S\in \mathbb{C}^{m\times d}$, we have that $q_{S,b}^A = \prod\limits_{s=1}^n q_{S,b_s}^A$ , where $q_{S,b_s}^A$ are coprime with each other. As a result, $r_{S,b}^A =  \sum \limits_{s=1}^n r_{S,b_s}^A$. 
	\end{lemma}
	
	\begin{proof} First,  for each $s$, we claim that $(S,A,b_s)$-annihilating polynomial forms an ideal generated by $q_{S,b}^A$, i.e.,
		\begin{equation}\label{ideal2}
		\mathcal{I}^{\mathrm{A}}_{S,b_s} = \{p \in \mathbb{C}[z]\ |\ Sp(A)\mathcal{K}_{r_{S}^{A}}(A,b_s) = \{0\}\} = \langle q_{S,b_s}^A\rangle.
		\end{equation} 
This is due to the fact (see the proof of Lemma 2.3 in \cite{aldroubi2014krylov} )that 

    $$Sq_{S,b_s}^A(A) \mathcal{K}_{r_{S}^{A}}(A,b_s) = \{0\}\iff  Sq_{S,b_s}^A(A) h(A)b_s= 0, \text{ for any polynomial } h(z). $$

	From the definition, it is straightforward to see that $q_{S,b_s}^A$ divides $q^A_{b_s}$. So by Lemma \ref{lemma1},  we have
		\begin{equation}\label{Smindegree}
		    q_{S,b_s}^A(z) = (z-\lambda_i)^{r_{S,b_s}^A}
		\end{equation}
		and it follows that $q_{S,b_s}^A$  is coprime to $q_{S,b_j}^A$ if $s\neq j$.
		
	   Now for each $s=1,\cdots,n$, let $p_s(A)=\prod\limits_{j=1\\j\neq s}^n q_{b_s}^A$. Then $p_s(A) b= p_s(A)b_s$. For any polynomial $h$, we have 
		\begin{equation}\label{keyrelation}
		    Sq_{S,b}^A(A) h(A)p_s(A)b_s =   Sq_{S,b}^A(A) h(A)p_s(A)b = 0
		\end{equation}
		holds true. So from the property of idea (\ref{ideal2}), we have $q_{S, b_s}^A$ divides $q_{S,f}^Ap_s(A)$. Since $q_{S, b_s}^A$ is coprime to $p_s(A)$, we have  $q_{S, b_s}^A$ divides $q_{S,b}^A$. Combining the fact that $q_{S,b_s}^A$  is coprime to $q_{S,b_j}^A$ if $s\neq j$, we have that $\prod\limits_{s=1}^n q_{S,b_s}^A$ divides $q_{S,b}^A$.  
		
		On the other hand, for $(S,A,b)$-annihilating polynomials, we have 	the ideal	
				\begin{equation}
		\mathcal{I}^{{A}}_{S,b} = \{p \in \mathbb{C}[z]\ |\ Sp(A)\mathcal{K}_{r_{S}^{A}}(A,b) = 0\} = \langle q_{S,b}^A\rangle.
		\end{equation}
		
Note that for any polynomial $h[z]$, we have
		\begin{equation}
		    S\prod \limits_{j=1}^n q_{S,b_j}^A(A)h(A)b = S\prod \limits_{j=1}^n q_{S,b_j}^A(A)h(A)\sum\limits_{s=1}^{n} b_s =\sum\limits_{s=1}^{n} S h(A)\prod \limits_{s=1,s\neq j}^n q_{S,b_s}^A(A)  q_{S,b_j}^A(A)b_j = 0,
		\end{equation}
		it follows that $q_{S,b}^A$ divides $\prod \limits_{s=1}^n q_{S,b_s}^A$.
		
		Therefore		\begin{equation}
		q_{S,b}^A = \prod\limits_{s=1}^n q_{S,b_s}^A. 
		\end{equation}
	\end{proof}

	% remark of lemma 2
    
	% lemma 3
	\begin{lemma}\label{lemma3}
	Let $b\in  \mathbb{C}^d$, $A\in \mathbb{C}^{d\times d}$ and $S\in\mathbb{C}^{m\times d}$. Given $r_{S,b}^A$, then $q_{S,b}^A$ is the unique monic polynomial $q$ satisfying $deg(q)\leq r_{S,b}^A $ and the following system of linear equations:
	\begin{equation}\label{minimalEquation}
		Sq(A)A^{t}b =0, t=0,\cdots,r_{S,b}^{A}-1.	\end{equation}
	\end{lemma}

	\begin{proof}

		First of all, we claim that the solutions to \eqref{minimalEquation} is the same with the solutions to the system of linear equations:
	   \begin{equation}\label{equivMinimalEquationEnumForm}
		Sq(A)A^{t}b=0, t=0,\cdots,r_{b}^{A}-1.	
\end{equation}

 Suppose that $q$ is a solution to \eqref{minimalEquation}, then for any $j\geq r_{S,b}^{A}$,
		
		$$A^{j}=p_j(A)q_{S,b}^A(A)+h_j(A), \mathrm{deg}(h_j(A)) \leq r_{S, b}^{A}-1. $$Therefore, 
\begin{align}		
Sq(A)A^{j}b&=S q(A)p_j(A)q_{f^*,b}^A(A) b+S q(A)  h_j(A)b\nonumber\\&=Sq_{S,b}^A(A)q(A)p_j(A)b+Sq(A)  h_j(A)b\\&=0+0=0,
\end{align} where we use the property of $q_{S,b}^A(A)$ in \eqref{keyrelation} to obtain that $Sq_{S,b}^A(A)q(A)p_j(A) b$=0. The claim is proved. Therefore, $q=q_{S,b}^A$ by deg$(q)\leq r_{S,b}^A$  and the definition of $q_{S,b}^A$. 
		
	\end{proof}

\begin {theorem}[Theorem 2.6 in \cite{aldroubi2017dynamical}]
\label {tadcul}
Let $J\in \mathbb{C}^{ d \times d}$ be a matrix in Jordan form  as in  \eqref{Jordan1}. 
Let   $\{f_i: i =1,\cdots, m\}\subset \mathbb{C}^d$ be a finite subset of vectors, recall that $r^J_{f_i}$ is the degree of  the $(J, f_i)$-annihilator, %of the vector $b_i$ 
  $l_i=r_i-1$, and $ P_J = \{P_s: s = 1, \cdots, n\}$ be the penthouse family for $J$ introduced in Definition \ref {defJord}. 

Then the following statements are equivalent:
\begin{enumerate}

\item[(i)]
The set of vectors $\{J^jf_i: \; i=1,\cdots,m, j=0,\dots, l_i\}$ spans $\mathbb{C}^d$.
\item[(ii)]
The set of vectors $\{P(\lambda_j;J)N^{l_j}f_i: \; i=1,\cdots,m, j=0,\dots, l_i\}$ spans $V_j$ for $j=1,\cdots,n$.
\item[(ii)]
For every $s = 1, \dots, n$, the set $\{P_s f_i,  i=1,\cdots,m\}$ spans  $E_s = P_s\mathbb{C}^n$.\end{enumerate}
\end{theorem}

\lstset{language=Matlab,%
    %basicstyle=\color{red},
    breaklines=true,%
    morekeywords={matlab2tikz},
    keywordstyle=\color{blue},%
    morekeywords=[2]{1}, keywordstyle=[2]{\color{black}},
    identifierstyle=\color{black},%
    stringstyle=\color{mylilas},
    commentstyle=\color{mygreen},%
    showstringspaces=false,%without this there will be a symbol in the places where there is a space
    numbers=left,%
    numberstyle={\tiny \color{black}},% size of the numbers
    numbersep=9pt, % this defines how far the numbers are from the text
    emph=[1]{for,end,break},emphstyle=[1]\color{red}, %some words to emphasise
    %emph=[2]{word1,word2}, emphstyle=[2]{style},    
}

%\section*{Matlab Code}

%\lstinputlisting{affinelinsys.m}\label{LIP_matlab}

\bibliographystyle{siam}	% (uses file "plain.bst")
\bibliography{refs.bib}		% expects file "myrefs.bib"

\end{document}